\documentclass[10pt]{amsart}

\usepackage[english]{babel}

\usepackage[all]{xy}
\usepackage{amssymb,amsmath}
\usepackage[mathscr]{euscript}
\usepackage{mathrsfs}

\newtheorem{thm}{Theorem}[section]
\newtheorem{prop}[thm]{Proposition}
\newtheorem{cor}[thm]{Corollary}
\newtheorem{lm}[thm]{Lemma}

\newtheorem*{thm*}{Theorem}

\theoremstyle{definition}
\newtheorem{df}[thm]{Definition}
\newtheorem{ex}[thm]{Example}

\newtheorem{num}[thm]{}
\newtheorem{paragr}[thm]{}

\theoremstyle{remark}
\newtheorem{rem}[thm]{Remark}

\numberwithin{equation}{thm}

\newcommand{\lis} {\mathscr Sm(k)}
\newcommand{\lisc} {\mathscr Sm^{cor}(k)}
\newcommand{\dmgme} {DM_{gm}^{eff}\!(k)}
\newcommand{\dmgm} {DM_{gm}\!(k)}

\newcommand{\dmm} {DM_{-}^{eff}\!(k)}
\newcommand{\DM} {DM(k)}

\newcommand{\C}{\mathscr C}

\newcommand{\HH}{\mathrm H}

\newcommand{\sHom} {\underline{\mathrm{Hom}}}
\newcommand{\Hom}{\operatorname{\mathrm{Hom}}}

\newcommand{\mor}[3] {\mathrm{Hom}_{#1} \! \left( #2,#3 \right)}

\newcommand{\mot}[1] {M\!\left( #1 \right)}
\newcommand{\motx}[2] {M_{#2}\!\left( #1 \right)}

\newcommand{\cohm}[3]{H^{#1}_{\mathcal M}(#3;\ZZ(#2))}

\newcommand{\spec}[1] {\mathrm{Spec}(#1)}

\newcommand{\tra}{{}^t}

\newcommand{\dual}[1]{#1^{*}}
\newcommand{\unit}{\mathbf 1}

\newcommand{\chern}[2] {\mathfrak c_{#1}(#2)}
\newcommand{\lef}[2] {\mathfrak l_{#1}(#2)}
\newcommand{\pur}[1] { \mathfrak p_{(#1)} }

\newcommand{\dte} { \mathbb A^1_k }

\newcommand{\dten}[1] { \mathbb A^{#1}_k }
\newcommand{\dtex}[1] { \mathbb A^1_{#1} }
\newcommand{\dtenx}[2] { \mathbb A^{#1}_{#2} }

\newcommand{\PP} { \mathbb P }

\newcommand{\ZZ} {\mathbb Z}
\newcommand{\QQ} {\mathbb Q}

\newcommand{\ppred} {\textbf{(Red)}~}
\newcommand{\ppmv} {\textbf{(MV)}~}
\newcommand{\pphtp} {\textbf{(Htp)}~}

\newcommand{\ppexc} {\textbf{(Exc)}~}
\newcommand{\ppadd} {\textbf{(Add)}~}

\newcommand{\ecupp} {\scriptstyle\boxtimes\textstyle\!}

\newcommand{\ecuppx}[1] {\scriptstyle\boxtimes_{#1}\textstyle\!}

\newcommand{\dtwist}[1]{\!((#1))}

\begin{document}

\title{Around the Gysin triangle {I}}

\author{Fr\'ed\'eric D\'eglise}
\address{LAGA\\
CNRS~(UMR 7539)\\
Universit\'e Paris~13\\
\hbox{Avenue~Jean-Baptiste~Cl\'ement}\\
93430 Villetaneuse\\France}
\email{deglise@math.univ-paris13.fr}
\urladdr{http://www.math.univ-paris13.fr/~deglise/}
\thanks{Partially supported by the ANR (grant No. ANR-07-BLAN-042)}

\date{2005 - last revision: 04/2011}

\begin{abstract}
%
%
We define and study Gysin morphisms on mixed motives over a perfect field.
 Our construction extends the case of closed immersions, 
 already known from results of Voevodsky, to arbitrary projective morphisms.
 We prove several classical formulas in this context,
  such as the projection and excess intersection formulas,
  and some more original ones involving residues.
 We give an application of this construction to duality
  and motive with compact support.
\end{abstract}

\maketitle

\section*{Introduction}

Since Poincar\'e discovers the first instance of duality in singular homology, mathematicians slowly became aware that most of cohomology theories could be equipped with an exceptional functoriality, covariant, usually referred to as either transfer, trace or more recently Gysin morphism\footnote{
 The term \emph{transfer} is more frequently used for finite morphisms,
  \emph{trace} for structural morphisms of projective smooth schemes over a field,
  and \emph{Gysin morphisms} for the zero section of a vector bundle, usually understand
  as a part of the Gysin long exact sequence.}.
 In homology, this kind of exceptional functoriality exists accordingly.
 The most famous case is the pullback on Chow groups.
Motives of Voevodsky are homological: they are naturally covariant. As they modeled homology theory, they should be equipped with an exceptional functoriality, contravariant. This is what we primarily prove here for smooth schemes over a field. Further,
 we focus on the two fundamental properties of Gysin morphisms: their functorial nature and their compatibility with the natural functoriality, corresponding to various projection formulas. The reader can already guess the intimate relationship of this theory with the classical intersection theory.

The predecessor of our construction was to be found in the Gysin triangle defined by Voevodsky\footnote{See \cite[chap. 5, Prop. 3.5.4]{FSV}.} 
for motives over a perfect field $k$: 
 associated with a closed immersion $i:Z \rightarrow X$ between smooth $k$-schemes, 
 Voevodsky constructs a distinguished triangle of mixed motives:
$$
M(X-Z) \rightarrow M(X) \xrightarrow{i^*} M(Z)(n)[2n]
 \xrightarrow{\partial_{X,Z}} M(X-Z)[1].
$$
The arrow labelled $i^*$ is the Gysin morphism associated with the closed immersion $i$.
Because this triangle corresponds to the so-called localization long exact sequence in cohomology,
fundamental in Chow and higher Chow theory, it has a central position in the theory of mixed motives. In \cite{Deg3} and \cite{Deg5bis}, we studied its naturality,
which corresponds to the projection formulas mentionned in the first paragraph,
 for the Gysin morphism $i^*$. 
 Interestingly, we discovered that these formulas had 
 counterpart for the \emph{residue morphism} $\partial_{X,Z}$ appearing in 
 the Gysin triangle\footnote{The reader is referred to section \ref{sec:base_change}
 for a summary of these results.}.
The main technical result 
 of this article (see Theorem \ref{thm:assoc}) is the functoriality property 
 of the Gysin morphism $i^*$.
 But, as in the case of projection formulas, 
 this comes with new formulas for the residue morphism. Let us quote it now:
\begin{thm*}
Let $X$ be a smooth $k$-scheme, $Y$ (resp. $Y'$) be a	smooth closed subscheme of $X$
of pure codimension $n$ (resp. $m$). 
Assume the reduced scheme $Z$ associated with $Y \cap Y'$
 is smooth of pure codimension $d$.
Put $Y_0=Y-Z$, $Y'_0=Y'-Z$, $X_0=X-Y \cup Y'$.

Then the following diagram,
with $i$,$j$,$k$,$l$,$i'$ the evident closed immersions,
 is commutative~:
$$
\xymatrix@R=16pt@C=24pt{
M(X)\ar^-{j^*}[r]\ar_{i^*}[d]\ar@{}|{(1)}[rd]
 & M(Y')(m)[2m]\ar^-{\partial_{X,Y'}}[r]\ar^{k^*}[d]
    \ar@{}^/8pt/{(2)}[rd]
 & M(X-Y')[1]\ar^{(i')^*}[d] \\
M(Y)(n)[2n]\ar_{l^*}[r]
 & M(Z)(d)[2d]\ar^/-5pt/{\partial_{Y,Z}}[r]
     \ar^{\partial_{Y',Z}}[d]\ar@{}|{(3)}[rd]
 & M(Y_0)(n][2n+1]\ar^{\partial_{X_0,Y_0}}[d] \\
& M(Y'_0)(m)[2m+1]\ar_-{-\partial_{X_0,Y'_0}}[r]
 & M(X_0)[2].
}
$$
\end{thm*}
\noindent This theorem can be understood as follows: the commutativity of square (1) in fact gives the functoriality of the Gysin morphism (take $Y'=Z$) ; the commutativity of square (2) shows the Gysin triangle is functorial with respect to the Gysin morphism of a closed immersion.
 Finally the commutativity of square (3) reveals the differential nature of the residue morphism: 
 it can be seen as an analogue of the change of variable theorem 
 for computing the residue of differential forms.\footnote{In fact, 
 one can show that the residue morphisms of motives induces the usual residue 
 on differential forms via De Rham realization.}

More generally, our Gysin morphism is associated with any morphism between smooth $k$-schemes. 
We go from the case of closed immersions to that of projective morphisms by a nowadays classical method\footnote{A model for us was the pullback on Chow groups as defined by Fulton in \cite{Ful}.}. Using the projective bundle formulas for motives, one easily defines the Gysin morphism for the projection of a projective bundle. As any projective morphism $f$ can be factored as closed immersion $i$ followed by the projection of a projective bundle $p$, we can put:
$f^*=p^*i^*$. The key point is to show this definition is independant of the factorisation. Taking into account the theorem cited above, this reduces to prove that for any section $s$ of a the projection $p$, the following relation holds: $p^*s^*=1$. When the definition is correctly settled,
the main properties of the general Gysin morphism follows from the particular case of closed immersions. Let us summarize them for the reader:
\begin{itemize}
\item functorial nature (Prop. \ref{Gysin_funct}),
\item projection formula in the transversal case (Prop. \ref{trivial_proj_form_Gysin}),
\item excess intersection formula (Prop. \ref{excess}),
\item naturality of the Gysin triangle with respect to Gysin morphisms
  (Prop. \ref{Gysin_morph&Gysin_tri}).
\end{itemize}

To end this description of the motivic Gysin morphism,
 we come back to the point of view at the beginning of the introduction.
 It was told that the existence of this exceptional functoriality was a consequence
 of Poincar\'e duality. In the end of this work, we go on the reverse side: Poincar\'e duality
 is a consequence of the existence of the Gysin morphism\footnote{Though stated in a
 different language, this was already observed and used in \cite[XVIII]{SGA4}.}.
 In fact, we use the tensor structure on the category of mixed motives
 and construct duality  pairings for a smooth projective $k$-scheme $X$ of dimension $n$.
Let $p:X \rightarrow \spec k$ (resp. $\delta:X \rightarrow X \times_k X$)
be the canonical projection (resp. diagonal embedding) of $X/k$. 
We obtain duality pairings (cf Theorem \ref{duality})
\begin{align*}
\eta: \ & \ZZ \xrightarrow{p^*} M(X)(-n)[-2n]
 \xrightarrow{\delta_*} M(X)(-n)[-2n] \otimes M(X) \\
\epsilon: \ & M(X) \otimes M(X)(-n)[-2n]
 \xrightarrow{\delta^*} M(X) \xrightarrow{p_*} \ZZ.
\end{align*}
which makes $M(X)(-n)[-2n]$ a \emph{strong dual} of $M(X)$
 in the sense of Dold-Puppe (see Paragr. \ref{num:df_duality} for recall on this notion).
 This result implies the usual formulation of Poincar\'e duality:
 the motivic cohomology of $X$ is isomorphic to its motivic homology 
 via cap-product with a homological class, the \emph{fundamental class of $X/k$.}
 But this duality result holds more universally: any motive defines both a cohomology
 and a homology ; the previous duality statement is valid in this
 generalized setting.

The morality of this result is that
 the existence of the Gysin morphism is essentiallty equivalent 
 to Poincar\'e duality when one restricts to projective smooth schemes over $k$
 (we left the precise statement to the reader).

\subsubsection*{Brief description of the organization of the paper}

Section 1 is concerned with the Gysin triangle associated with a closed immersion.
Sections 1.1, 1.2 and 1.3 contains reminders on the articles \cite{Deg3} and \cite{Deg5bis},
 concerning both the definitions and the results.
 Section 1.4 contains the proof of the main theorem of this paper, as stated above.

In section 2, we develop the general Gysin morphism: 
 section 2.1 contains essentially the proof that the definition explained above 
 is independant of the choice of the factorization, 
 section 2.2 states and proves the properties listed above.
 In the end of section 2.2, we also relate our Gysin morphism in the
 case of finite \'etale covers with the transfers one gets using
 the theory of finite correspondences (see Prop. \ref{prop:Gysin=transpose}).
 Section 2.3 explores duality as explained above,
 and shows how one can deduce a natural construction of a motive with compact support.

\subsubsection*{Further background and references}

Gysin morphisms for motives were already constructed by M.~Levine within his
 framework of mixed motives in \cite{Lev}.\footnote{Recall Levine showed an equivalence
 of triangulated monoidal categories between his category of mixed motives and the
 one of Voevodsky under the assumption of resolution of singularities.}
 The treatment of Levine has common feature with ours. In comparison, our principal contribution
 consists in the formula involving residues, together with the excess intersection formula.
The construction of Gysin morphism on cohomology --
 which follows from its existence on motives through realization --
 was also treated directly by Panin in the setting of oriented cohomologies.
 In his setting, Panin does not consider residue morphisms.

This work has been available as a preprint for a long time.\footnote{
It first appears on the preprint server of the LAGA in 2005.} 
It has been used in \cite{BVK} by Barbieri-Viale and Kahn about questions of duality. 
Ivorra refers to it in \cite{Ivo} mainly concerning \emph{motivic fundamental classes}
 (Def. \ref{df:motivic_fdl_class} here). Our initial interest for the Gysin morphism
 was motivated by the some computation in the coniveau filtration at the level
 of motives ; we refer the reader to \cite{Deg11} in this book for this subject.

We have extended the considerations of the present paper in a more general setting in \cite{Deg8}:
 the base can be arbitrary
 and we work in an abstract setting which allows to consider both motives and $MGL$-modules
 -- the latter corresponds to generalized oriented cohomologies, see \emph{loc. cit.}
 for details.
 The present version is still useful as the proof are much simpler.
Let us mention also the fundamental work \cite{Ayoub} of Ayoub on cross functors.
It yields Gysin morphisms through a classical procedure (dating back to \cite{SGA4}).
However, one has to take care about questions of orientation
which are not treated by Ayoub (aka Thom isomorphisms). This is done in
\cite{CD3}. On the other hand,
 the excess intersection formula, as well as formulas involving residues do not
 follow directly from the 6 functors formalism but from the analysis done here.

A final word concerning Poincar\'e duality: it was well known that strong duality
 for motives of smooth projective $k$-schemes was a consequence of the construction by
 Voevodsky of a $\otimes$-functor from Chow motives to geometric motives
 (see \cite[chap. 5, 2.1.4]{FSV}).
 On the other hand, our direct proof of duality shows the existence of this functor
 (see Remark \ref{duality&Chow})
 without using the theory of Friedlander and Lawson on moving cycles
  (\cite{FL}).\footnote{Explicitly:
 the proof of Prop. 2.1.4 of \cite[chap. 5]{FSV} refers to \cite[chap. 4, 7.1]{FSV}
 which uses in particular \cite[chap. 4, 6.3]{FSV} whose proof is a reference
 to \cite{FL}.} Let us mention also that the new idea in our definition of the motive
 with compact support of a smooth $k$-scheme is that the Gysin morphism of the diagonal 
 allows to construct a comparison functor from the motive with compact support to
 the usual motive (see property (iv) after Def. \ref{df:motif_c})
 -- this idea was already used in \cite{CD1}. Compared to other versions of
 motive with compact support, one by Voevodsky in \cite[chap. 5, \textsection 4]{FSV} 
 and the other by Huber-Kahn in \cite[app. B]{HK}, 
 ours allows one to bypass the
 assumptions of resolution of singularities for some of the fundamental properties.

\section*{Acknowledgments}

This paper grew out of a non published part of my thesis
 and I want to thank my thesis director F.~Morel 
 for several useful discussions.
 I also want to thank D.C.~Cisinski, B.~Kahn and J.~Wildeshaus
 for useful exchanges and encouragement during the writing of this
 article. Special thanks go to C.~Weibel and U.~Jannsen for pointing
 to me the sign issue in the formula numbered (3) in the introduction.
 Finally, I want to thank the referee for comments
 which have helped me to clarify and improve the present redaction.

\section*{Notations and conventions}

We fix a base field $k$ which is assumed to be perfect. 
The word scheme will stand for any
separated $k$-scheme of finite type, and we will say that a scheme is
smooth when it is smooth over the base field. 
The category of smooth schemes
is denoted by $\lis$.
Throughout the paper, when we talk about 
the codimension of a closed immersion,
the rank of a projective bundle 
or the relative dimension of a morphism,
we assume it is constant.

Given a vector bundle $E$ over $X$, and $P$ the associated projective
 bundle with projection $p:P \rightarrow X$, 
 we will call \emph{canonical line bundle} on $P$ the canonical invertible sheaf
 $\lambda$ over $P$ characterized by the property that $\lambda \subset p^{-1}(E)$.
 Similarly, we will call \emph{canonical dual line bundle} on $P$
 the dual of $\lambda$.

We say that a morphism is \emph{projective} if it admits a factorization
 into a closed immersion followed by the projection of a projective
 bundle.\footnote{Beware this is not the convention of \cite{EGA2} unless
 the aim of the morphism admits an ample line bundle.}

We let $\dmgm$ (resp. $\dmgme$) be the category of geometric motives 
(resp. effective geometric motives) introduced in \cite[chap. 5]{FSV}.
For the result of section 1, we work in the category $\dmgme$.
If $X$ is a smooth scheme, we denote by $\mot X$ the effective motive 
 associated with $X$ in $\dmgme$.
From section 2 to the end of the article, we work in the category
 $\dmgm$. Then $M(X)$ will be the motive associated with $X$ in
 the category $\dmgm$ 
 (through the canonical functor $\dmgme \rightarrow \dmgm$).

For a morphism $f:Y \rightarrow X$ of smooth schemes, we will simply put
$f_*=\mot f$. Moreover for any integer $r$, 
we sometimes put $\ZZ\,\dtwist r=\ZZ(r)[2r]$ in large diagrams.
When they are clear from the context (for example in diagrams), 
 we do not indicate twists or shifts on morphisms.

\tableofcontents

\section{The Gysin triangle}

\subsection{Relative motives}

\begin{df}
\label{df:pair}
We call closed (resp. open) pair any couple $(X,Z)$ (resp. $(X,U)$)
such that $X$ is a smooth scheme and $Z$ (resp. $U$) is a closed
(resp. open) subscheme of $X$. 

Let $(X,Z)$ be an arbitrary closed pair. We will say $(X,Z)$ is
 smooth if $Z$ is smooth. For an integer $n$, we will say
 that $(X,Z)$ has codimension $n$
 if $Z$ has (pure) codimension $n$ in $X$.

A morphism of open or closed pairs $(Y,B) \rightarrow (X,A)$
is a couple of morphisms $(f,g)$ which fits into the commutative diagram
of schemes
$$
\xymatrix@=10pt{
B\ar@{^(->}[r]\ar_g[d] & Y\ar^f[d] \\
A\ar@{^(->}[r] & X.
}
$$
If the pairs are closed, we also require that this square
 is topologically cartesian\footnote{\emph{i.e.} cartesian 
 as a square of topological spaces ; in other words,
 $B_{red}=(A \times_X Y)_{red}$.}.

We add the following definitions~:
\begin{itemize}
\item The morphism $(f,g)$ is said to be cartesian if the above square
is cartesian as a square of schemes.
\item A morphism $(f,g)$ of closed pairs is said to be excisive if
$f$ is {\'e}tale and $g_{red}$ is an isomorphism.
\item A morphism $(f,g)$ of smooth closed pairs is said to be transversal
 if it is cartesian and the source and target have the
 same codimension.
\end{itemize}
\end{df}

We will denote conventionally open pairs as fractions $(X/U)$.

\begin{df}
\label{rel_motiv}
Let $(X,Z)$ be a closed pair. We define the relative motive 
$\motx X Z$ --- sometimes denoted by $\mot{X/X-Z}$ ---
associated with $(X,Z)$ to be the class in $\dmgme$ of the complex
$$
... \rightarrow 0 \rightarrow [X-Z] \rightarrow [X]
 \rightarrow 0 \rightarrow ...
$$ 
\noindent where $[X]$ is in degree $0$.
\end{df}

Relative motives are functorial with respect to morphisms of closed pairs.
In fact, $\motx X Z$ is functorial with respect to
morphisms of the associated open pair $(X/X-Z)$. For example,
if $Z \subset T$ are closed subschemes of $X$, we get a morphism
$\motx X T \rightarrow \motx X Z$.

If $j:(X-Z) \rightarrow X$ denotes the complementary open immersion,
 we obtain a canonical distinguished triangle in $\dmgme$~:
\begin{equation}
\label{eq:gysin_beta}
\mot{X-Z} \xrightarrow{j_*} \mot X
 \rightarrow \motx X Z \rightarrow \mot{X-Z}[1].
\end{equation}

\begin{rem}
The relative motive in $\dmgme$ defined here
corresponds under the canonical embedding to
 the relative motive in $\dmm$ defined in \cite[def. 2.2]{Deg3}.
\end{rem}
 
The following proposition sums up the basic properties 
of relative motives.
It follows directly from \cite[1.3]{Deg3} using the previous remark.
Note moreover that in the category $\dmgme$, each property is 
rather clear, except \ppexc which follows from the embedding
theorem \cite[chap. 5, 3.2.6]{FSV} of Voevodsky.
\begin{prop} \label{properties_motx}
Let $(X,Z)$ be a closed pair. 
The following properties of relative motives hold:
\begin{enumerate}
\item[\ppred] \emph{Reduction}: If we denote by
 $Z_0$ the reduced scheme associated with $Z$ then: 
$$\motx{X}{Z}=\motx{X}{Z_0}.$$
\item[\ppexc] \emph{Excision}: If $(f,g):(Y,T) \rightarrow (X,Z)$ is
an excisive morphism then $(f,g)_*$ is an isomorphism.
\item[\ppmv] \emph{Mayer-Vietoris}~: If $X=U \cup V$ is an open
 covering of $X$ then we obtain a canonical distinguished triangle
 of shape:
\begin{align*}
\motx {U \cap V}{Z \cap U \cap V}
& \xrightarrow{\mot{j_U}-\mot{j_V}}
   \motx{U}{Z \cap U} \oplus \motx{V}{Z \cap V} \\
&  \xrightarrow{\mot{i_U}+\mot{i_V}} \motx{X}{Z}
 \longrightarrow \motx {U \cap V}{Z \cap U \cap V}[1].
\end{align*}
The morphism $i_U$, $i_V$, $j_U$, $j_V$ stands for the obvious
cartesian morphisms of closed pairs induced by the corresponding
canonical open immersions.
\item[\ppadd] \emph{Additivity}: Let $Z'$ be a closed subscheme of
$X$ disjoint from $Z$. Then the morphism induced by the inclusions
$$
\motx{X}{Z \sqcup Z'} \rightarrow \motx{X}{Z} \oplus \motx{X}{Z'}
$$
is an isomorphism.
\item[\pphtp] \emph{Homotopy}: Let $\pi:(\dtex X,\dtex Z)
\rightarrow (X,Z)$ denote the cartesian morphism induced by the
projection. Then $\pi_*$ is 
an isomorphism.
\end{enumerate}
\end{prop}

\subsection{Purity isomorphism} \label{sec:pur}

\begin{num} \label{motivic_chern}
Consider an integer $i \geq 0$.
Recall that the $i$-th twisted motivic complex over $k$ is defined
 according to Voevodsky as
 \emph{Suslin's singular simplicial complex}
 of the cokernel of the natural map of sheaves with
 transfers $\ZZ^{tr}(\dten i-0) \rightarrow \ZZ^{tr}(\dten i)$,
 shifted by $2i$ degrees on the left
 (cf \cite{SV} or \cite{FSV}).
Motivic cohomology of a smooth scheme $X$ in degree $n \in \ZZ$ and twists $i$
is defined following Beilinson's idea as the Nisnevich hypercohomology groups
of this complex which we denote by $\cohm n i X$.
Moreover,
 there is a natural pairing of complexes
  $\ZZ(i) \otimes \ZZ(j) \rightarrow \ZZ(i+j)$
 (cf \cite{SV}) which induces the product on motivic cohomology.

Recall there exists\footnote{\label{note:iso_V}
Following Voevodsky, 
this isomorphism is obtained from the Nisnevich hypercohomology spectral sequence
of the complex $\ZZ(i)$ once we have observed that
$\HH^q(\ZZ(i))=0$ if $q>i$ and $\HH^i(\ZZ(i))$ is canonically isomorphic with
the $i$-th Milnor unramified cohomology sheaf $\mathcal K_i^M$. 
The compatibility
with product and pullback then follows from a careful study (cf for example 
\cite[8.3.4]{Deg1}).}
a canonical isomorphism
\begin{equation} \label{motivic_coh&Chow}
\epsilon_X:CH^i(X) \xrightarrow{\ \sim \ } \cohm {2i} i X
\end{equation}
which is functorial with respect to pullbacks and compatible with products.

According to \cite[chap. 5, 3.2.6]{FSV}, we also get an isomorphism
\begin{equation}\label{cohm_in_DMgm}
\cohm n i X \simeq \mor{\dmgme}{\mot X}{\ZZ(i)[n]}
\end{equation}
where $\ZZ(i)$ on the right hand side stands (by the usual abuse of notation)
for the $i$-th Tate geometric motive.
In what follows, we will identify cohomology classes in motivic
cohomology with morphisms in $\dmgme$ according to this isomorphism.

Thus cup-product on motivic cohomology corresponds
 to a product on morphisms that we describe now.
Let $X$ be a smooth scheme, 
$\delta:X \rightarrow X \times_k X$ be the diagonal embedding
and $f:\mot X \rightarrow \mathcal M$, $g:\mot X \rightarrow \mathcal N$ 
be two morphisms with target a geometric motive.
We define the \emph{exterior product} of $f$ and $g$, 
denoted by $f \ecuppx X g$ or simply $f \ecupp g$, as the composite
\begin{equation}\label{eq:df:ecupp}
\mot X \xrightarrow{\delta_*} \mot X \otimes \mot X
 \xrightarrow{f \otimes g} \mathcal M \otimes \mathcal N.
\end{equation}
In the case where $\mathcal M=\ZZ(i)[n]$, $\mathcal N=\ZZ(j)[m]$,
identifying the tensor product 
$\ZZ(i)[n] \otimes \ZZ(j)[m]$ with $\ZZ(i+j)[n+m]$
by the canonical isomorphism, the above product corresponds exactly 
to the cup-product on motivic cohomology.

According to the isomorphism \eqref{motivic_coh&Chow},
motivic cohomology admits Chern classes.
Thus, applying the isomorphism \eqref{cohm_in_DMgm},
we attach to any vector bundle $E$ on a smooth scheme $X$
and any integer $i \geq 0$,
 the following morphism in $\dmgme$
\begin{equation} \label{Chern_motivic}
\chern i E:\mot X \rightarrow \ZZ(i)[2i]
\end{equation}
which corresponds under the preceding isomorphisms
 to the $i$-th Chern class of $E$ in the Chow group.
 For short, we call this morphism the \emph{$i$-th motivic Chern class}
 of $E$.
\end{num}

\begin{rem} \label{rem:Chern_classes_products&pullbacks}
According to our construction, any formula in the Chow group
involving pullbacks and intersections of Chern classes induces
a corresponding formula for the morphisms of type \eqref{Chern_motivic}.
\end{rem}

\begin{num}
\label{proj_bdle_iso}
We finally recall the projective bundle theorem 
(cf \cite[chap. 5, 3.5.1]{FSV}).
Let $P$ be a projective bundle of rank $n$ over a smooth scheme $X$,
$\lambda$ its canonical dual line bundle
and $p:P \rightarrow X$ the canonical projection.
The projective bundle theorem of Voevodsky says that the morphism
\begin{equation} \label{eq:chom&proj_bdl_iso}
\mot P \xrightarrow{\sum_{i \leq n} {\chern 1 \lambda}^i \ecupp p_*}
 \bigoplus_{i=0}^n \mot X\dtwist i
\end{equation}
is an isomorphism.

Thus, we can associate with $P$ a family of split monomorphisms indexed
by an integer $r \in [0,n]$ corresponding to the decomposition of its
motive~:
\begin{equation} \label{eq:lefschetz4proj}
\lef r {P}:M(X)(r)[2r] \rightarrow
 \oplus_{i \leq n} M(X)(i)[2i]
 \rightarrow \mot P.
\end{equation}
The following lemma will be a key point in the theory of the Gysin morphism:
\end{num}
\begin{lm} \label{lm:key_computation_cycle_class}
Consider the notations introduced above.

Let $x \in CH^n(P)$ be a cycle class
 and $x_i \in CH^{n-i}(X)$ be cycle classes such that
\begin{equation} \label{eq:Chow&proj_bdl}
x=\sum_{i=0}^n p^*(x_i).c_1(\lambda)^i.
\end{equation}
Consider an integer $i \in[0,n]$ and the following morphisms in $\dmgme$
\begin{align*}
\mathfrak x&:\mot X \rightarrow \ZZ(n)[2n] \\
\mathfrak x_i&:\mot X \rightarrow \ZZ(n-i)[2(n-i)]
\end{align*}
associated respectively with $x$ and $x_i$ through the isomorphisms
 \eqref{motivic_coh&Chow} and \eqref{cohm_in_DMgm}.

Then we get the equality of morphisms $\mot X(i)[2i] \rightarrow \ZZ(r)[2r]$
 in $\dmgme$:
$$
\mathfrak x \circ \lef i P=\mathfrak x_i(i)[2i].
$$
\end{lm}
\begin{proof}
Taking care of Remark \ref{rem:Chern_classes_products&pullbacks},
 the equality \eqref{eq:Chow&proj_bdl} induces the following
 equality of morphisms $\mot P \rightarrow \ZZ(r)[2r]$:
$$
\mathfrak x=\sum_{i=0}^r \mathfrak c_1(\lambda)^i \ecupp (\mathfrak x_i \circ p_*)
=\sum_{i=0}^r \big\lbrack\mathfrak x_i(i)[2i]\big\rbrack  \circ \mathfrak c_1(\lambda)^i \ecupp p_*.
$$
The second equality follows from the definition of the exterior cup product
 (formula \eqref{cohm_in_DMgm}). Thus, the definition of $\lef i P$
 and the formula \eqref{eq:chom&proj_bdl_iso}
 for the projective bundle isomorphism on motives 
 allow to conclude.
\end{proof}

\begin{rem} \label{rem:particular_cancel}
Note in particular that we deduce from the preceding lemma the following
 weak form of the cancellation theorem of Voevodsky \cite{V2}:
 for any smooth scheme $X$ and
  any non negative integers $(n,i)$ such that $i \leq n$, the morphism
\begin{align*}
\Hom_{\dmgme}(\mot X,\ZZ(n-i)[2(n-i)])
 &\rightarrow \Hom_{\dmgme}(\mot X(i)[2i],\ZZ(n)[2n]), \\
\phi &\mapsto \phi(i)[2i]
\end{align*}
is an isomorphism.
\end{rem}

\begin{lm} \label{lm:Thom&Proj_split}
Let $X$ be a smooth scheme and $E/X$ be a vector bundle.
Consider the projective completion $P$ of $E/X$,
 the closed pair $(P,X)$ corresponding to the canonical section of $P/X$
 and the complementary open immersion $j:U \rightarrow P$.
Then the distinguished triangle \eqref{eq:gysin_beta} 
associated with $(P,X)$
\begin{equation} \label{eq:Thom&Proj_split}
\mot U \xrightarrow{j_*} \mot P \xrightarrow{\pi_P} \motx P X
 \rightarrow \mot U[1]
\end{equation}
is split.
\end{lm}
\begin{proof}
Recall $P=\PP(E \oplus \dtex X)$. Let $\nu:\PP(E) \rightarrow P$
 be the embedding associated with the monomorphism of vector bundles
 $E \rightarrow E \oplus \dtex X$.
 The closed immersion $\nu$ factors through the open immersion
  $j:U \rightarrow P$.
 Let us denote finally by $L$ the canonical line bundle on $\PP(E)$
 and by $s_0$ its zero section.
 Then, according to \cite[\textsection 8]{EGA2}, there exists an isomorphism
 of schemes
 $\epsilon:L \rightarrow U$ such that the following diagram commutes:
$$
\xymatrix@=20pt{
L\ar^{\epsilon}[r] & U\ar^j[d] \\
\PP(E)\ar^-\nu[r]\ar^{s_0}[u] & P.
}
$$
Thus the morphism $j_*$ is isomorphic in $\dmgme$ to the morphism
$$
\nu_*:\mot{\PP(E)} \rightarrow \mot P
$$
which is a split monomorphism according to the 
respective projective bundle isomorphisms for $\PP(E)/X$ and $P/X$.
\end{proof}

\num \label{def_diag}
Consider a smooth closed pair $(X,Z)$.
Let $N_ZX$ (resp. $B_ZX$) be the normal bundle (resp. blow-up) of $(X,Z)$
and $P_ZX$ be the projective completion of $N_ZX$.
We denote by $B_Z(\dtex X)$
the blow-up of $\dtex X$ with center $\{0\} \times Z$.
It contains as a closed subscheme the trivial blow-up 
$\dtex Z=B_Z(\dtex Z)$. 
We consider the closed pair $(B_Z(\dtex X),\dtex Z)$ over 
$\dte$. Its fiber over $1$ is the closed pair $(X,Z)$
and its fiber over $0$ is $(B_ZX \cup P_ZX,Z)$.
Thus we can consider the following deformation diagram~:
\begin{equation} \label{1st_def_diag}
(X,Z) \xrightarrow{\bar \sigma_1} (B_Z(\dtex X),\dtex Z)
 \xleftarrow{\bar \sigma_0} (P_ZX,Z).
\end{equation}
This diagram is functorial in $(X,Z)$ with respect
to cartesian morphisms of closed pairs. Note finally that,
on the closed subschemes of each closed pair,
$\bar \sigma_0$ (resp. $\bar \sigma_1$) is the $0$-section
(resp. $1$-section) of $\dtex Z/Z$.

The existence statement in the following proposition appears
already in \cite[2.2.5]{Deg5bis} but the uniqueness statement is
new~:
\begin{prop}
\label{prop:purity}
Let $n$ be a natural integer.

There exists a unique family of isomorphisms of the form
$$
\pur{X,Z}:\motx X Z \rightarrow M(Z)(n)[2n]
$$
indexed by smooth closed pairs of codimension $n$ such that~:
\begin{enumerate}
\item for every cartesian morphism $(f,g):(Y,T) \rightarrow (X,Z)$ of
smooth closed pairs of codimension $n$, the following diagram is
commutative~:
$$
\xymatrix@R=20pt@C=40pt{
\motx Y T\ar^{(f,g)_*}[r]\ar_{\pur{Y,T}}[d]
 & \motx X Z\ar^{\pur{X,Z}}[d] \\
M(T)(n)[2n]\ar^-{g_*(n)[2n]}[r] &  M(Z)(n)[2n].
}
$$
\item Let $X$ be a smooth scheme and $P$ be the projective completion
of a vector bundle $E/X$ of rank $n$.
Consider the closed pair $(P,X)$ corresponding to the $0$-section of $E/X$. 
Then $\pur{P,X}$ is the inverse of the following morphism
$$
M(X)(n)[2n]
 \xrightarrow{\lef n P} \mot P
 \xrightarrow{\pi_P} \motx P X.
$$
where $\lef n P$ is the monomorphism of \eqref{eq:lefschetz4proj}
 and $\pi_P$ is the epimorphism
 of the split distinguished triangle \eqref{eq:Thom&Proj_split}.
\end{enumerate}
\end{prop}
\begin{proof}
Uniqueness~: Consider a smooth closed pair
$(X,Z)$ of codimension $n$.

Applying property (1) to the deformation diagram \eqref{1st_def_diag},
 we obtain the commutative diagram~:
$$
\xymatrix@R=30pt{
\mot{X,Z}\ar^-{\bar \sigma_{1*}}[r]\ar_{\pur{X,Z}\!}[d]
 & \mot{B_Z(\dtex X),\dtex Z}\ar|{\phantom{\big(}}[d]
 \ar@{}|{\pur{B_Z(\dtex X),\dtex Z}}[d]
 & \mot{P_ZX,Z}\ar^{\pur{P_ZX,Z}}[d]
                           \ar_-{\bar \sigma_{0*}}[l] \\
M(Z)(n)[2n]\ar^-{s_{1*}}[r]
 & M(\dtex Z)(n)[2n]
 & M(Z)(n)[2n]\ar_-{s_{0*}}[l] 
}
$$
Using homotopy invariance, $s_{0*}$ and $s_{1*}$ are isomorphisms. 
Thus in this diagram, all the morphisms are
isomorphisms. Now, the second property of the purity isomorphisms
determines uniquely $\pur{P_ZX,Z}$, thus $\pur{X,Z}$ is
also uniquely determined.

For the existence part, we refer the reader to \cite{Deg5bis},
section 2.2. 
\end{proof}

\begin{rem}
The second point of the above proposition appears as a
normalization condition. It will be reinforced later
 (cf Remark \ref{strong_normalization}).
\end{rem}

\begin{df}
\label{df:gysin_triangle}
Let $(X,Z)$ be a smooth closed pair of codimension $n$. 
Denote by $j$ (resp. $i$) the open immersion
$(X-Z) \rightarrow X$ (resp. closed immersion $Z \rightarrow X$).

With the notation of the preceding proposition,
 the morphism $\pur{X,Z}$ 
 will be called the \emph{purity isomorphism} associated with $(X,Z)$.
 
Using this isomorphism, we deduce from the
distinguished triangle \eqref{eq:gysin_beta} the following
distinguished triangle in $\dmgme$,
 called the Gysin triangle of $(X,Z)$
$$
\mot{X-Z} \xrightarrow{j_*} \mot X \xrightarrow{i^*}
 M(Z)(n)[2n] \xrightarrow{\partial_{X,Z}} \mot{X-Z}[1].
$$
The morphism $\partial_{(X,Z)}$ (resp. $i^*$)
is called the \emph{residue} (resp. \emph{Gysin morphism})
associated with $(X,Z)$ (resp. $i$). Sometimes we use the notation
$\partial_i=\partial_{(X,Z)}$.
\end{df}

\begin{ex}
\label{Gysin_proj_bdle}
Consider a smooth scheme $X$ and a vector bundle $E/X$ of rank $n$. 
Let $P$ be the projective completion of $E$, 
 $\lambda$ be its canonical dual invertible sheaf
  and $p:P \rightarrow X$ be its canonical projection.
Consider the canonical section 
$s:X \rightarrow P$ of $P/X$.

We define the Thom class of $E$ in $CH^n(P)$ as the class
$$
t(E)=\sum_{i=0}^n p^*(c_{n-i}(E)).c_1(\lambda)^i.
$$
It corresponds according to paragraph \ref{motivic_chern}
 to a morphism
 $\mathfrak t(E):\mot P \rightarrow \ZZ(n)[2n]$.

Consider the notations of Lemma \ref{lm:Thom&Proj_split}
 together with the definition of the exterior product
 \eqref{eq:df:ecupp}.
 By definition of Chern classes,
  the restriction of the class $t(E)$ to $\PP(E)$
  is zero. Because the canonical map $\PP(E) \rightarrow U$
  is a homotopy equivalence\footnote{See the argument in the proof
  of Lemma \ref{lm:Thom&Proj_split}.},
  we get that $j^*(t(E))=0$.
 Thus, as the triangle \eqref{eq:Thom&Proj_split} is split,
 the morphism
$$
\mathfrak t(E) \ecuppx P p_*:\mot P \rightarrow \mot X(n)[2n]
$$
factors uniquely through $\pi_P$:
$$
\mot P \xrightarrow{\pi_P} \motx P X
 \xrightarrow{\epsilon_P} \mot X(n)[2n].
$$
Because the coefficient of $c_1(\lambda)^n$ in $t(E)$ is $1$,
 we deduce from Lemma \ref{lm:key_computation_cycle_class}
 that $\epsilon_P \circ \pur{P,X}^{-1}=1$.
Thus, according to the previous definition,
 we obtain the following formula\footnote{
This is the analog of the well-known formula in Chow theory:
for any cycle class $x \in CH^*(Z)$,
$s_*(x)=t(E).p^*(x)$.}:
\begin{equation} \label{eq:Thom&Gysin}
s^*=\mathfrak t(E) \ecuppx P p_*.
\end{equation}
\end{ex}

\begin{rem}
\label{Gysin&Voevodsky}
Our Gysin triangle agrees with that of \cite{FSV}, chap. 5,
prop. 3.5.4. Indeed, in the proof of 3.5.4, Voevodsky constructs an
isomorphism which he denotes by $\alpha_{(X,Z)}$. He then uses it
 as we use the purity isomorphism to construct his triangle. 
It is not hard to check that
this isomorphism $\alpha_{(X,Z)}$ satisfies the two conditions of Proposition
\ref{prop:purity} and thus coincides with the purity isomorphism 
from the uniqueness statement.
\end{rem}

\subsection{Base change formulas} \label{sec:base_change}

This subsection is devoted to recall some results we obtained previously
in \cite{Deg3} and \cite{Deg5bis} about the following type of
morphism~:
\begin{df}
\label{refined_morphism}
Let $(X,Z)$ (resp. $(Y,T)$) be
 a smooth closed pair of codimension $n$ (resp. $m$). 
Let $(f,g):(Y,T) \rightarrow (X,Z)$ be a morphism of closed pairs.

We define the morphism $(f,g)_!$ as the following composite~:
$$
M(T)(m)[2m] \xrightarrow{\ \pur{Y,T}^{-1}\ } \mot{Y,T}
 \xrightarrow{(f,g)_*} \mot{X,Z} \xrightarrow{\ \pur{X,Z}\ }
  M(Z)(n)[2n].
$$
\end{df}

In the situation of this definition, let $i:Z \rightarrow X$
 and $k:T \rightarrow Y$ be the obvious closed embeddings
 and $h:(Y-T) \rightarrow (X-Z)$ be the restriction of $f$.
Then we obtain from our definitions the following commutative diagram~: 
\begin{equation}
\begin{split}
 \label{eq:refined_formulas}
\xymatrix@R=22pt@C=22pt{
\mot{Y-T}\ar[r]\ar[d]
 & \mot Y\ar^-{j^*}[r]\ar_{f_*}[d]\ar@{}|/-8pt/{_{(1)}}[rd]
 & M(T)(m)[2m]\ar^-{\partial_{Y,T}}[r]\ar|{(f,g)_!}[d]
       \ar@{}|/-8pt/{_{(2)}}[rd]
 & \mot{Y-T}[1]\ar^{h_*}[d] \\
\mot{X-Z}\ar[r]
 & \mot X\ar^-{i^*}[r]
 & M(Z)(n)[2n]\ar^-{\partial_{X,Z}}[r] & \mot{X-Z}[1] }
\end{split}
\end{equation}
The commutativity of square \scriptsize (1) \normalsize corresponds
to a \emph{refined projection formula}. The word refined is inspired
by the terminology ``refined Gysin morphism'' of Fulton in
\cite{Ful}. By contrast, the commutativity of square \scriptsize
(2) \normalsize involves motivic cohomology rather than Chow groups.


\begin{num} \label{num:e_thick}
Let $T$ (resp. $T'$) be a	closed subscheme of
a scheme $Y$ with defining ideal $\mathcal J$
 (resp. $\mathcal J'$).
We will say that a closed immersion $i:T \rightarrow T'$ is
 an \emph{exact thickening of order $r$ in $Y$}
 if $\mathcal J'=\mathcal J^r$.
We recall to the reader the following formulas
 obtained in \cite[3.1, 3.3]{Deg3}~:
\end{num}
\begin{prop}
\label{prop:gysin_transversal}
Let $(X,Z)$ and $(Y,T)$ be smooth closed pairs of codimension $n$ and
$m$ respectively. Let $(f,g):(Y,T) \rightarrow (X,Z)$ be a morphism of
closed pairs.
\begin{enumerate}
\item \emph{(Transversal case)} If $(f,g)$ is transversal (which implies $n=m$)
 then
 
$(f,g)_!=g_*(n)[2n]$.
\item \emph{(Excess intersection)} If $(f,g)$ is cartesian, we put
  $e=n-m$ and $\xi=g^*N_ZX/N_TY$. Then 
  $(f,g)_!=\chern e \xi \ecuppx T  g_*(m)[2m]$.
\item \emph{(Ramification case)} If $n=m=1$ and the canonical
closed immersion $T \rightarrow Z \times_X Y$ is an exact thickening 
of order $r$ in $Y$, then $(f,g)_!=r.g_*(1)[2]$.
\end{enumerate}
\end{prop}
Note that each case of the above proposition gives,
 via the commutative Diagram \eqref{eq:refined_formulas}, two formulas: 
 one involving Gysin morphisms
 and the other one involving the residues. 
 When we will apply this proposition, we will always refer to one these
 two formulas.

\begin{rem}
In the article \cite[4.23]{Deg8},
 the case (3) has been generalized to any codimension $n=m$. 
In this generality,
the integer $r$ is simply the geometric multiplicity of $Z \times_X Y$
 -- when assumed to be connected.
\end{rem}

\begin{cor} \label{cor:Gysin&open_closed}
Let $X$ be a smooth scheme such that $X=X_1 \sqcup X_2$.
Consider the open and closed immersion
 $\nu_i:X_i \rightarrow X$ for $i=1,2$.

Then the isomorphism 
$(\nu_{1*},\nu_{2*}):\mot{X_1} \oplus \mot{X_2} \rightarrow \mot X$
admits as an inverse isomorphism the map 
$(\nu_{1}^*,\nu_{2}^*):\mot X \rightarrow \mot{X_1} \oplus \mot{X_2}$.
\end{cor}
\begin{proof}
In fact, according to the first point of the above proposition,
 we get the following relations for $i=1,2$:
$\nu_i^*\nu_{i*}=1$, $\nu_{2-i}^*\nu_{i*}=0$.
This, together with the fact $(\nu_{1*},\nu_{2*})$ is an isomorphism,
 allows to conclude.
\end{proof}

Another application of the preceding proposition 
 is the following projection formula:
\begin{cor}
\label{cor:proj_gysin}
Let $(X,Z)$ be a smooth pair of codimension $n$ and $i:Z
\rightarrow X$ be the corresponding closed immersion.

Then, $(1_Z \ecupp_Z i_*) \circ i^*=i^* \ecupp_X 1_X:
\mot X \rightarrow \mot Z \otimes \mot X(n)[2n]$.
\end{cor}
\begin{proof}
Just apply point (1) of the proposition to the
cartesian morphism $(X,Z) \rightarrow (X \times X,Z \times
X)$ induced by the diagonal embedding of $X$. The only thing left to
check is that $(i \times 1_X)^*=i^* \otimes 1$, which was done in
\cite[2.6.1]{Deg5bis}.
\end{proof}

\begin{rem}
In the above statement, we have loosely identified
the motive
$\mot Z \otimes \mot X(n)[2n]$ with
$(\mot Z(n)[2n]) \otimes \mot X$ through the canonical isomorphism. 
This will not have any consequences in the present article.
On the contrary in \cite{Deg5bis},
 we must be attentive to this isomorphism which may result
  in a change of sign (cf remark 2.6.2 of {\it loc. cit.}).
\end{rem}

Another corollary of the preceding proposition is the following
 analog of the self-intersection formula:
\begin{cor}
Let $(X,Z)$ be a smooth closed pair of codimension $n$
 with normal bundle $N_ZX$.
If $i$ denotes the corresponding closed immersion,
 we obtain the following equality:
$$
i^*i_*=\mathfrak c_n(N_ZX) \ecuppx Z 1_{Z*}.
$$
\end{cor}
\noindent Indeed it follows from the transversal case of
 the preceding proposition applied to the 
 cartesian morphism $(i,1_Z):(Z,Z) \rightarrow (X,Z)$ and
 from the commutativity of square $(1)$
  in diagram \eqref{eq:refined_formulas}.

\begin{ex}
Consider a vector bundle $p:E \rightarrow X$ of rank $n$.
Let $s_0$ be its zero section.
According to the
homotopy property in $\dmgme$, we get $s_{0*} p_*=1$.
Thus, the preceding corollary applied to $s_0$ implies
the following formula:
\begin{equation} \label{eq:Euler&Gysin}
s_0^*=\mathfrak c_n(p^{-1}E) \ecuppx E p_*.
\end{equation}
Moreover, the Gysin triangle associated with $s_0$
 together with the isomorphism $s_{0*}$
 gives the following distinguished triangle:
$$
\mot{E^\times} \longrightarrow \mot E
 \xrightarrow{\mathfrak c_n(E) \ecuppx X 1_{X*}} M(X)(n)[2n]
 \xrightarrow{\partial_{E,X} \circ s_{0*}} \mot{E^\times}[1]
$$
where $E^\times$ is the complement of the zero section.
Following a classical terminology,
 we call it the \emph{Euler triangle} of $E/X$.\footnote{
It is the analog of the Euler long exact sequence 
associated with $E/X$ in cohomology.}
\end{ex}

\begin{df} \label{df:motivic_fdl_class}
Let $(X,Z)$ be a smooth closed pair of codimension $n$
 and $i:Z \rightarrow X$ be the corresponding closed immersion.
Let $\pi:Z \rightarrow \spec k$ be the structural morphism of $Z$.

We define the \emph{motivic fundamental class} of $Z$ in $X$
 as the following composite map:
$$
\eta_X(Z):\mot X \xrightarrow{i^*} \mot Z(n)[2n]
 \xrightarrow{\pi_*} \ZZ(n)[2n].
$$
\end{df}

\begin{ex}
Let $X$ be a smooth scheme 
 and $p:E \rightarrow X$ be a vector bundle of rank $n$.
According to formula \eqref{eq:Euler&Gysin},
 the motivic fundamental class of the zero section of $E/X$ is:
\begin{equation} \label{equation:fdl_class&Euler_class}
\eta_E(X)=\mathfrak c_n(p^{-1}E).
\end{equation}
Let $P/X$ be the projective completion of $E/X$.
According to formula \eqref{eq:Thom&Gysin},
 the motivic fundamental class of the canonical section of $P/X$ is:
\begin{equation}\label{equation:fdl_class&Thom_class}
\eta_P(X)=\mathfrak t(E).
\end{equation}
\end{ex}

\begin{rem}
If we use the cancellation theorem of Voevodsky
 (see \cite{V2} or use more directly Remark \ref{rem:particular_cancel}),
 the Gysin map $i^*$ induces a canonical pushout\footnote{
 We prove in \cite[lem. 3.3]{Deg9} that this pushout coincides
 through the isomorphism \eqref{motivic_coh&Chow} with the usual
 pushout in Chow theory.}:
$$
i_*:\cohm s t Z \rightarrow \cohm{s+2n} {t+n} X.
$$
Then, through the isomorphism \eqref{cohm_in_DMgm},
 we get the equality $\eta_X(Z)=i_*(1)$, where $1$ stands
 for the unit of the (bigraded) cohomology ring $\cohm * * Z$.
 This motivates our terminology.
\end{rem}

According to the computations of the previous example, 
 the following lemma
 is a generalization of formulas \eqref{eq:Thom&Gysin}
 and \eqref{eq:Euler&Gysin}:
\begin{lm} \label{lm:Gysin&fdl_class_retraction}
Let $(X,Z)$ be a smooth closed pair of codimension $n$
 and $i:Z \rightarrow X$ be the corresponding closed immersion.
 Assume that $i$ admits a retraction $p:X \rightarrow Z$.

Then $i^*=\eta_X(Z) \ecuppx X p_*$.
\end{lm}
\begin{proof}
Let $\pi:Z \rightarrow \spec k$ be the structural morphism.
According to formula \eqref{eq:df:ecupp}, we deduce that
 $\pi_* \ecuppx Z 1_{Z*}=1_{Z*}$.
The lemma follows from the following computation:
\begin{align*}
i^*&
\stackrel{(1)}=\lbrack \pi_* \ecuppx Z (p_*i_*) \rbrack \circ i^*
=(\pi_* \otimes p_*)(1_{Z*} \ecuppx Z i_*) \circ i^* 
\stackrel{(2)}=(\pi_* \otimes p_*)(i^* \ecuppx X 1_{Z*}) \\
&=\eta_X(Z) \ecuppx X p_*
\end{align*}
where equality $(1)$ is justified by the preceding remark and 
 the relation $pi=1_Z$ whereas equality (2) is in fact
 Corollary \ref{cor:proj_gysin}.
\end{proof}

\begin{lm} \label{lm:Gysin&fdl_class}
Let $X$ be a smooth scheme
 and $E/X$ be a vector bundle of rank $n$.
 Let $s$ (resp. $s_0$) be a section (resp. the zero section)
 of $E/X$. Assume that $s$ is transversal to $s_0$
 and consider the cartesian square:
$$
\xymatrix@=14pt{
Z\ar^i[r]\ar_k[d] & X\ar^s[d] \\
X\ar^{s_0}[r] & E
}
$$
Then the motivic fundamental class of $i$ is:
$$
\eta_X(Z)=\mathfrak c_n(E).
$$
\end{lm}
\begin{proof}
Let $\pi$ (resp. $\pi'$) be the structural morphism of $Z$ (resp. $X$).
The lemma follows from the computation below:
\begin{align*}
\eta_X(Z)=\pi_*i^*=\pi'_*k_*i^*
\stackrel{(1)}=\pi'_*s_0^*s_*
\stackrel{(2)}=\mathfrak c_n(p^{-1}E) \circ s_*
\stackrel{(3)}=\mathfrak c_n(E) \circ p_* \circ s_*=\mathfrak c_n(E).
\end{align*}
Equality (1) follows from Proposition \ref{prop:gysin_transversal},
 equality (2) from Formula \eqref{equation:fdl_class&Euler_class}
 and equality (3) from Remark \ref{rem:Chern_classes_products&pullbacks}.
\end{proof}

\begin{ex} \label{ex:excess&fdl_class}
Let $E/X$ be a vector bundle and
 $p:P \rightarrow X$ be its projective completion.
Let $\lambda$ be the canonical dual line bundle on $P$.
Put $F=\lambda \otimes_P p^{-1}(E)$ as a vector bundle over $P$.
According to our conventions,
 we get a canonical embedding $\lambda^\vee \subset p^{-1}(E \oplus \dtex X)$.
Then the following composite map
$$
\lambda^\vee \rightarrow p^{-1}(E \oplus \dtex X)
 \rightarrow p^{-1}(E)
$$
corresponds to a section $\sigma$ of $F/P$. 
One can check that $\sigma$ is transversal to the zero section $s_0^F$
of $F/P$ and that the following square is cartesian:
$$
\xymatrix@=14pt{
X\ar^s[r]\ar[d] & P\ar^\sigma[d] \\
P\ar^{s_0^F}[r] & F
}
$$
where $s$ is the canonical section of $P/X$.
Thus the preceding corollary gives the following equality:
$\eta_P(X)=\mathfrak c_n(F)$.\footnote{In fact,
 from the definition of the Thom class (Example \ref{Gysin_proj_bdle}),
 one can check directly the equality
 $c_n(F)=t(E)$ in the Chow group $CH^n(P)$:
 the computation we get in this example shows 
 that our (sign) conventions are coherent.}
\end{ex}

\subsection{Composition of Gysin triangles}

We first establish lemmas needed for the main theorem.
First of all, using the projection formula in the transversal
case (cf \ref{prop:gysin_transversal}) and the compatibility
of Chern classes with pullbacks,
 we easily obtain the following result:
\begin{lm}
\label{lm:proj&Gysin_mot}
Let $(Y,Z)$ be a smooth closed pair of codimension $m$ and $P/Y$ be
a projective bundle of dimension $n$. We put $V=Y-Z$ 
and consider the following cartesian squares~:
$$
\xymatrix@C=18pt@R=15pt{
P_V\ar^\nu[r]\ar_{p_{V}}[d] & P\ar_{p}[d]
 & P_Z\ar_\iota[l]\ar^{p_{Z}}[d] \\
V\ar^j[r] & Y & Z\ar_i[l]
}
$$
Finally, we consider the canonical line bundle $\lambda$
(resp. $\lambda_V$, $\lambda_Z$) on $P$ (resp. $P_V$,
$P_Z$).

Then, for any integer $r \in [0,n]$, the following diagram is 
commutative
$$
\xymatrix@R=36pt@C=24pt{
\mot{P_V}\ar^{\nu_*}[r]
    \ar@{}|{\chern 1 {\lambda_V}^r \ecupp p_{V*} \quad }[d]
    \ar|{\phantom{\big(}}[d]
 & \mot{P}\ar^-{\iota^*}[r]
    \ar@{}|{\chern 1 \lambda^r \ecupp p_* \quad }[d]
    \ar|{\phantom{\big(}}[d]
 & \mot{P_Z}\dtwist m\ar^-{\partial_\iota}[r]
    \ar@{}|{\chern 1 {\lambda_Z}^r \ecupp p_{Z*} \quad }[d]
    \ar|{\phantom{\big(}}[d]
 & \mot{P_V}[1]
    \ar@{}|{\chern 1 {\lambda_V}^r \ecupp p_{V*}[1]}[d]
    \ar|{\phantom{\big(}}[d] \\
\mot V\dtwist r\ar^-{j_*}[r]
 & \mot Y\dtwist r\ar^-{i^*}[r]
 & \mot Z\dtwist {r+m}\ar^-{\partial_i}[r]
 & \mot V\dtwist r[1].
}
$$
\end{lm}

The following lemma will be in fact the crucial case in the proof 
 of the next theorem.
\begin{lm}
\label{lm:proj&Gysin_mot2}
Let $X$ be a smooth scheme
 and $E/X$ (resp. $E'/X$) be a vector bundle of rank $n$ (resp. $m$).
 Let $P$ (resp. $P'$) be the projective completion of $E/X$
 (resp. $E'/X$) and $i$ (resp. $i'$) its canonical section.

We put $R=P \times_X P'$
 and consider the closed immersions:
$$
i:X \rightarrow P, j:P \rightarrow R, k=j \circ i,
$$
where $j=P \times_X i'$ and $k=(i,i')$.
Then $k^*=i^* j^*$.
\end{lm}
\begin{proof}
We consider the following canonical morphisms:
$$
\xymatrix@=18pt{
R\ar^q[r]\ar_{q'}[d]\ar@{}|\pi[rd]\ar|/-0.3pt/{\phantom{\Big(}}[rd]
 & P'\ar^{p'}[d] \\
P\ar_p[r] & X
}
$$
According to Lemma \ref{lm:Gysin&fdl_class_retraction},
 we obtain
\begin{equation*} \label{eq:pf_lm:proj&Gysin_mot2_bis}
i^*=\eta_P(X) \ecupp_P p_*, \quad
j^*=\eta_R(P) \ecupp_R q'_*, \quad
k^*=\eta_R(X) \ecupp_P \pi_*.
\end{equation*}

Applying the first case of Proposition \ref{prop:gysin_transversal}
 to the cartesian morphism of closed pairs $(q',p'):(R,P') \rightarrow (P,X)$,
 we obtain the relation:
$$
\eta_P(X) \circ q'_*=\eta_R(P').
$$
Together with the preceding computations,
 it implies the following equality:
$$
i^*j^*=\eta_R(P) \, \ecupp_P \eta_R(P') \, \ecupp_P \pi_*.
$$
Thus we are reduced to prove the relation:
\begin{equation} \label{eq:pf_lm:proj&Gysin_mot2}
\eta_R(X)=\eta_R(P) \ecuppx R \eta_R(P').
\end{equation}

Consider the notations of Example \ref{ex:excess&fdl_class}
 applied to the case of $E/X$ (resp. $E'/X$):
 we get a vector bundle $F/P$ (resp. $F'/P$) of rank $n$ (resp. $m$)
 such that:
\begin{align*}
&\eta_P(X)=\mathfrak c_n(F), \\
\text{resp. } & \eta_{P'}(X)=\mathfrak c_m(F').
\end{align*}
Let $\sigma$ (resp. $\sigma'$) be the section of $F/P$
 (resp. $F'/P'$) constructed in \emph{loc. cit.}
Consider the vector bundle over $R$ defined as:
$$
G=F \times_X F'=q^{\prime-1}(F) \oplus q^{-1}(F').
$$
We get a section $(\sigma \times_X \sigma')$ of $G/P$
 which is transversal to the zero section $s_0^G$
 and such that the following square is cartesian:
$$
\xymatrix@=18pt{
X\ar^i[r]\ar[d] & R\ar^{\sigma \times_X \sigma'}[d] \\
R\ar^{s_0^G}[r] & G.
}
$$
Thus, according to Lemma \ref{lm:Gysin&fdl_class},
 we obtain:
$$
\eta_R(X)=\mathfrak c_{n+m}(G).
$$
The relation \eqref{eq:pf_lm:proj&Gysin_mot2}
 now follows from Remark \ref{rem:Chern_classes_products&pullbacks}
 and the equality
$$
c_{n+m}(G)=q^{\prime*}(c_n(F)).q^*(c_m(F'))
$$
in $CH^{n+m}(R)$.
\end{proof}

\begin{thm}
\label{thm:assoc}
Consider a topologically cartesian square of smooth schemes
$$
\xymatrix@=10pt{
Z\ar^k[r]\ar_l[d] & Y'\ar^j[d] \\
Y\ar^i[r] & X
}
$$
such that $i$,$j$,$k$,$l$ are closed immersions of respective pure
codimensions $n$, $m$, $s$, $t$. We put $d=n+t=m+s$
 and let $i':(Y-Z) \rightarrow (X-Y')$, $j':(Y'-Z) \rightarrow (X-Y)$
 be the closed immersion respectively induced by $i$, $j$.

Then the following diagram is commutative~:
$$
\xymatrix@R=20pt@C=30pt{
\mot X\ar^-{j^*}[r]\ar_{i^*}[d]\ar@{}|{(1)}[rd]
 & \mot{Y'}\dtwist m\ar^-{\partial_j}[r]\ar^{k^*}[d]
    \ar@{}^{(2)}[rd]
 & \mot{X-Y'}[1]\ar^{(i')^*}[d] \\
\mot{Y}\dtwist n\ar_{l^*}[r]
 & \mot Z\dtwist d\ar^/-5pt/{\partial_l}[r]
     \ar^{\partial_k}[d]\ar@{}|{(3)}[rd]
 & \mot{Y-Z}\dtwist n[1]\ar^{\partial_{i'}}[d] \\
& \mot{Y'-Z}\dtwist m[1]\ar_-{-\partial_{j'}}[r]
 & \mot{X-Y \cup Y'}[2]
}
$$
\end{thm}
\begin{proof}
We will simply call smooth triple the data $(X,Y,Y')$
 of a triple of smooth schemes $X$, $Y$, $Y'$ 
 such that $Y'$ and $Y$ are closed subschemes of $X$. 
 Such smooth triples form a category with morphisms 
 the commutative diagrams
$$
\xymatrix@=20pt{
\bar Y\ar_g[d]\ar@{^(->}[r] & \bar X\ar_f[d] & \bar Y'\ar^{g'}[d]\ar@{_(->}[l] \\
Y\ar@{^(->}[r] & X & Y'\ar@{_(->}[l]
}
$$
made of two cartesian squares. We say in addition that the morphism
$(f,g,g')$ is \emph{transversal} if $f$ is transversal to $Y$, $Y'$
 and $Y \cap Y'$.

To such a triple, we associate a geometric motive $\mot{X,Y,Y'}$
as the cone of the canonical map of complexes of $\lisc$
$$
\xymatrix@C=20pt@R=10pt{
\hdots\ar[r] & [X-Y \cup Y']\ar[r]\ar[d] & [X-Y']\ar[r]\ar[d] & \hdots \\
\hdots\ar[r] & [X-Y]\ar[r] & [X]\ar[r] & \hdots
}
$$
where $[X]$ and $[X-Y']$ are placed in degree $0$.
This motive is evidently functorial
 with respect to morphisms of smooth triples.

We will also use the notation
$\mot{\frac{X/X-Y}{X-Y'/X-Y \cup Y'}}$ for this motive because it is more suggestive.
By definition, it fits into the following diagram, with $\Omega=Y \cup Y'$:
$$
\xymatrix@C=12pt@R=16pt{
(\mathcal D):\mot{X-\Omega}\ar[r]\ar[d]
 & \mot{X-Y}\ar[r]\ar[d]
 & \mot{\frac{X-Y}{X-\Omega}}\ar[r]\ar[d]
 & \mot{X-\Omega}[1]\ar[d] \\
\mot{X-Y'}\ar[r]\ar[d]
 & \mot X\ar[r]\ar[d]\ar@{}|{(1)}[rd]
 & \mot{\frac X {X-Y'}}\ar[d]\ar[r]\ar@{}|{(2)}[rd]
 & \mot{X-Y'}[1]\ar[d] \\
\mot{\frac{X-Y'}{X-\Omega}}\ar[r]\ar[d]
 & \mot{\frac X {X-Y}}\ar[r]\ar[d]
 & \mot{\frac{X/X-Y}{X-Y'/X-\Omega}}
    \ar[r]\ar[r]\ar[d]\ar@{}|{(3)}[rd]
 & \mot{\frac{X-Y'}{X-\Omega}[1]}\ar[d] \\
\mot{X-\Omega}[1]\ar[r]
 & \mot{X-Y}[1]\ar[r]
 & \mot{\frac{X-Y}{X-\Omega}}[1]\ar[r]
 & \mot{X-\Omega}[2].
}
$$
In this diagram, every
square is commutative except square (3) which is anticommutative
due to the fact the permutation isomorphism on 
$\ZZ[1] \otimes \ZZ[1]$ is equal to $-1$. 
Moreover, any line or row of this diagram is a distinguished triangle.

With the hypothesis of the theorem, the proof will consist in
constructing a purity isomorphism  
$\pur{X,Y,Y'}:\mot{X,Y,Y'} \rightarrow M(Z)(d)[2d]$ which
satisfies the following properties~:
\begin{enumerate}
\item[(i)] \textit{Functoriality~:} The morphism $\pur{X,Y,Y'}$ 
is functorial
with respect to transversal morphisms of smooth triples. 
\item[(ii)] \textit{Symmetry~:} The following diagram is 
commutative~:
$$
\xymatrix@R=10pt@C=20pt{
\mot{X,Y,Y'}\ar_/-5pt/{\pur{X,Y,Y'}}[rd]\ar^{}[rr] &
 & \mot{X,Y',Y}\ar^/-5pt/{\quad \pur{X,Y',Y}}[ld] \\
& M(Z)(d)[2d]
}
$$
where the horizontal map is the canonical isomorphism.
\item[(iii)] \textit{Compatibility~:} The following diagram 
is commutative~:
$$
\xymatrix@C=13pt@R=20pt{
\mot{\frac{X-Y'}{X-\Omega}}\ar[r]\ar|{\pur{X-Y',Y-Z}}[d]
 & \mot{\frac X {X-Y}}\ar[r]\ar|{\pur{X,Y}}[d]
 & \mot{X,Y,Y'}\ar[r]
     \ar|{\pur{X,Y,Y'}}[d]
 & \mot{\frac{X-Y'}{X-\Omega}}[1]\ar|{\pur{X-Y',Y-Z}[1]}[d] \\
\mot{Y-Z}\dtwist n\ar[r]
 & \mot{Y}\dtwist n\ar^-{j^*}[r]
 & \mot Z\dtwist d\ar^-{\partial_j}[r]
 & \mot{Y-Z}\dtwist n[1]
}
$$
\end{enumerate}
With this isomorphism, we can deduce the three relations of the
theorem by considering squares $(1)$, $(2)$, $(3)$ in the above
diagram and applying the evident purity isomorphism where 
it belongs.

We then are reduced to construct the isomorphism and to prove the
above relations. 
The second relation is the most difficult one because we
have to show that two isomorphisms in a triangulated category
are equal. 
This forces us to be very precise in the construction
 of the isomorphism.

\bigskip

\noindent \underline{Construction of the purity isomorphism for smooth
  triples}~:

Consider the deformation diagram \eqref{1st_def_diag} for the closed
pair $(X,Y)$ and put $B=B_Y(\dtex X)$, $P=P_YX$.
Put also $(U,V)=(X-Y',Y-Z)$, $B_U=B \times_X U$ and $P_V=P \times_Y V$.
Note that, because $Z=(Y \times_X Y')_{red}$, we get $V=Y \times_X U$ ;
 thus $B_U$ is the deformation space of \eqref{1st_def_diag} for the
 closed pair $(U,V)$.
By functoriality of the deformation diagram and of relative motives
we obtain the following morphisms of distinguished triangles~:
$$
\xymatrix@C=26pt@R=8pt{
\mot{U,V}\ar[r]\ar[d]
 & \mot{X,Y}\ar[r]\ar[d]
 & \mot{\frac{X/X-Y}{U/U-V}}\ar^/16pt/{+1}[r]\ar[d] & \\
\mot{B_U,\dtex V}\ar[r]
 & \mot{B,\dtex Y}\ar[r]
 & \mot{\frac{B/B-\dtex Y}{B_U/B_U-\dtex V}}\ar^/16pt/{+1}[r] & \\
\mot{P_V,V}\ar[r]\ar[u]
 & \mot{P,Y}\ar[r]\ar[u]
 & \mot{\frac{P/P-Y}{P_V/P_V-V}}\ar^/16pt/{+1}[r]\ar[u] &
}
$$
According to Proposition \ref{prop:purity} and homotopy invariance,
 the vertical maps in the first two columns are isomorphisms.
As the rows in the diagram are distinguished triangles,
the vertical maps in the third column also are isomorphisms.

Using Lemma \ref{lm:proj&Gysin_mot} with $P=\PP(N_YX \oplus \dtex Y)$,
we can consider the following morphism of distinguished
triangles~:
$$
\xymatrix@C=26pt@R=16pt{
\mot{P_V,V}\ar[r]
 & \mot{P,Y}\ar[r]
 & \mot{\frac{P/P-Y}{P_V/P_V-V}}\ar^/16pt/{+1}[r] & \\
\mot{P_V}\ar[r]\ar[u]
 & \mot{P}\ar[r]\ar[u]
 & \mot{\frac{P}{P_V}}\ar^/16pt/{+1}[r]\ar[u]
 & \\
\mot{P_V}\ar[r]\ar@{=}[u]
 & \mot{P}\ar[r]\ar@{=}[u]
 & \mot{P_Z}\dtwist s\ar^/18pt/{+1}[r]
     \ar_/-3pt/{\pur{P,P_Z}^{-1}}[u]
 & \\
\mot{Y-Z}\dtwist n\ar[r]\ar^{\lef n {P_V}}[u]
 & \mot Y\dtwist n\ar[r]\ar^{\lef n {P}}[u]
 & \mot Z\dtwist d\ar^/14pt/{+1}[r]
    \ar_{\lef n {P_Z}}[u]
 &
}
$$
The triangle on the bottom is obtained by tensoring the Gysin triangle
of the pair $(Y,Z)$ with $\ZZ(n)[2n]$. 
From Proposition \ref{prop:purity}, 
 the first two of the vertical composite arrows are isomorphisms,
  so the last one is also an isomorphism.

If we put together (vertically) the two previous diagrams,
we finally obtain the following isomorphism of triangles~:
$$
\xymatrix@C=2pt@R=20pt{
\mot{U,V}\ar[r]\ar|{\pur{X-Y',Y-Z}}[d]
 & \mot{X,Y}\ar[rrr]\ar|{\pur{X,Y}}[d]
 &&& \mot{X,Y,Y'}\ar[rr]
     \ar^{(*)}[d]
 && \mot{U,V}[1]\ar[d] \\
\mot{Y-Z}\dtwist n\ar[r]
 & \mot Y\dtwist n\ar^/1pt/{j^*}[rrr]
 &&& \mot Z\dtwist d\ar^/-8pt/{\partial_j}[rr]
 && \mot{Y-Z}\dtwist n[1].
}
$$
We define $\pur{X,Y,Z}$ as the morphism labeled $(*)$ in
the previous diagram so that property (iii) follows from the construction.
The functoriality property (i) follows easily from
the functoriality of the deformation diagram.

\bigskip

\noindent \underline{The remaining relation}

To conclude it only remains to prove the symmetry property (ii). First of
all, we remark that the above construction implies immediately the
commutativity of the following diagram~:
$$
\xymatrix@R=10pt@C=6pt{
\mot{\frac{X/X-Y}{X-Y/X-Y \cup Y'}}\ar_{\pur{X,Y,Y'}}[rd]
  \ar^{}[rr]
& & \mot{\frac{X/X-Y}{X-Z/X-Y}}\ar^{\pur{X,Y,Z}}[ld] \\
& \mot Z\dtwist d, &
}
$$
where the horizontal map is induced by the evident open immersions.

Thus, it will be sufficient to prove the commutativity of the
following diagram~:
$$
\xymatrix@R=10pt@C=6pt{
\mot{\frac{X}{X-Z}}\ar_/-7pt/{\pur{X,Z}}[rd]\ar^{\alpha_{X,Y,Z}}[rr]
 & \ar@{}|/-3pt/{(*)}[d]
 & \mot{\frac{X/X-Y}{X-Z/X-Y}}\ar^/-7pt/{\quad \pur{X,Y,Z}}[ld] \\
& \mot Z\dtwist{n+m},
}
$$
where $\alpha_{X,Y,Z}$ denotes the canonical isomorphism. 

\bigskip

From now on, we consider only the smooth triples $(X,Y,Z)$ such that
$Z$ is a closed subscheme of $Y$.
Using the functoriality of $\pur{X,Y,Z}$, we remark that the
diagram {$(*)$} is natural with respect to morphisms $f:X' \rightarrow
X$ which are transversal to $Y$ and $Z$.

Consider the notations of the paragraph \ref{def_diag} and
put $D_ZX=B_Z(\dtex X)$ for short.
We will expand these notations as follows~:
$$
D(X,Z)=D_ZX, \; B(X,Z)=B_ZX \;, P(X,Z)=P_ZX\;, N(X,Z)=N_ZX.
$$
To $(X,Y,Z)$, we associate the evident closed pair $(D_ZX,D_ZX|_Y)$
 and the {\it double deformation space} 
$$D(X,Y,Z)=D(D_ZX,D_ZX|_Y).$$
This scheme is in fact fibered over $\dten 2$. 
The fiber over $(1,1)$ is $X$ and the fiber
over $(0,0)$ is $B(B_ZX \cup P_ZX,B_ZX|_Y \cup P_ZX|_Y)$.
In particular, the $(0,0)$-fiber contains the scheme $P(P_ZX,P_ZY)$.

We now put 
$
\left\{
\begin{array}{ll}
D=D(X,Y,Z), & R=P(R_ZX,R_ZY) \\
D'=D(Y,Y,Z), & P=R_ZY.
\end{array}
\right.
$ \\
Remark also that $D(Z,Z,Z)=\dtenx 2 Z$
and that
$R=P \times_Z P'$ where $P'=P_YX|_Z$.\footnote{The
last property is equivalent to the identification:
 $N(N_ZX,N_ZY)=N_ZY \oplus N_YX|_Z$.} 
From the description of the fibers of $D$ given above,
 we obtain a deformation diagram of smooth triples~:
$$
(X,Y,Z) \rightarrow (D,D',\dtenx 2 Z) \leftarrow (R,P,Z).
$$
Note that these morphisms are on the smaller closed subscheme the
$(0,0)$-section and $(1,1)$-section of $\dtenx 2 Z$ over $Z$, denoted 
respectively by $s_0$ and $s_1$.
Now we apply these morphisms to the diagram $(*)$ 
 in order to obtain the following commutative diagram~:
$$
\xymatrix@R=20pt@C=-16pt{
\motx X Z\ar|/-15pt/{\pur{X,Z}}[dd]\ar^/5pt/{\!\!\alpha_{X,Y,Z}}[rd]\ar[rr]
 && \motx D {\dtenx 2 Z}\ar|/-15pt/{\pur{D,\dtenx 2 Z}}[dd]\ar[rd]
 && \motx R Z\ar|/-15pt/{\pur{R,Z}}[dd]\ar^/5pt/{\!\!\alpha_{R,P,Z}}[rd]\ar[ll]
 \\
& \mot{X,Y,Z}\ar|/-0pt/{\quad \pur{X,Y,Z}}[ld]\ar[rr] &
 & \mot{D,D',\dtenx 2 Z}\ar|/-0pt/{\quad \pur{D,D',Z}}[ld] &
 & \mot{R,P,Z}\ar|/-0pt/{\quad \pur{R,P,Z}}[ld]\ar[ll]
 \\
\mot Z\dtwist{n+m}\ar@<-3pt>@{}_-{{s_1}_*}[rr]\ar[rr]
 && \mot{\dtenx 2 Z}\dtwist{n+m}
 && \mot Z\dtwist{n+m}.\ar^-{{s_0}_*}[ll]
}
$$
One knows that every part of this diagram save the triangle ones
 are commutative.
As the morphisms ${s_1}_*$ and ${s_{0}}_*$ are isomorphisms, the
commutativity of the left triangle is equivalent to the
 commutativity of the right one.

Thus, we are reduced to the case of the smooth triple $(R,P,Z)$.
Now, using the canonical split epimorphism
 $\mot{R} \rightarrow \motx R Z$, we are reduced to prove
the commutativity of the diagram~:
$$
\xymatrix@R=-6pt@C=50pt{
\mot{R}\ar_{i^*}[dd]\ar[rd] & \\
 & \mot{\frac{R/R-P}{R-Z/R-P}}
   \ar^/-0pt/{\quad \pur{R,P,Z}}[ld] \\
\mot Z\dtwist{d}
}
$$
where $i:Z \rightarrow R$ denotes the canonical closed immersion.

Using the property (iii) of the isomorphism $\pur{R,P,Z}$, 
we are finally reduced to prove the commutativity of the triangle
$$
\xymatrix@R=-2pt@C=50pt{
\mot{R}\ar_{i^*}[dd]\ar^{j^*}[rd] & \\
 & \mot P\dtwist n\ar^{k^*}[ld] \\
\mot Z\dtwist{d} &
}
$$
where $j$ and $k$ are the evident closed embeddings. 
This is Lemma \ref{lm:proj&Gysin_mot2}.
\end{proof}

As a corollary (take $j=i \circ l$, $k=1_Z$),
 we get the functoriality of the Gysin morphism of a closed immersion~:
\begin{cor}
\label{cor:assoc}
Let $Z \xrightarrow l Y \xrightarrow i X$ be closed immersion between
smooth schemes such that $i$ is of pure codimension $n$.

Then, $l^* \circ i^*=(i \circ l)^*$.
\end{cor}

As an illustration of the formulas obtained in the preceding theorem,
 we prove the following result:
\begin{prop} \label{prop:add_Gysin}
Consider a smooth closed pair $(X,Z)$ of codimension $n$
 and $\nu:Z \rightarrow X$ the corresponding immersion.

Consider the canonical decompositions $Z=\sqcup_{i \in I} Z_i$
 and $X=\sqcup_{j \in J} X_j$ into connected components.
 Put $\hat Z_j=Z \times_X X_j$.
 For any index $i \in I$, let $j \in J$ be the unique
 element such that $Z_i \subset X_j$ ;
 we let $\nu_i^j:Z_i \rightarrow X_j$
  be the immersion induced by $\nu$
  and 
 we denote by $Z'_i$ the unique scheme such that:
  $\hat Z_j=Z_i \sqcup Z'_i$.

Consider the following commutative diagram:
$$
\xymatrix@C=55pt{
\mot X\ar^{\nu^*}[r]
 & \mot Z\dtwist n\ar^{\partial_{X,Z}}[r]
 & \mot{X-Z}[1] \\
\oplus_{j \in J} \mot{X_j}\ar_-{(\nu_{ji})_{j \in J,i \in I}}[r]\ar^\sim[u]
 & \oplus_{i \in I} \mot{Z_i}\dtwist n\ar_-{(\partial_{ij})_{i \in I, j \in J}}[r]\ar_\sim[u]
 & \oplus_{j \in J} \mot{X_j-\hat Z_j}[1]\ar_\sim[u] \\
}
$$
where the vertical maps are the canonical isomorphisms.

Then, for any couple $(i,j) \in I \times J$,
\begin{enumerate}
\item if $Z_i \subset X_j$, $\nu_{ji}=\big(\nu_i^j\big)^*$
 and $\partial_{ij}=\partial_{X_j-Z'_i,Z_i}$,
\item otherwise,
$\nu_{ji}=0$  and $\partial_{ij}=0$.
\end{enumerate}
\end{prop}
\begin{proof}
We consider the following cartesian squares
 made of the evident immersions:
\begin{equation} \label{eq:pf:prop:add_Gysin}
\begin{split}
\begin{array}{ll}
\text{If } Z_i \subset X_j, &
\text{otherwise,} \\
\xymatrix{
Z_i\ar^{\nu_i^j}[r]\ar@{=}[d]
 & X_j\ar_{x_j}[d]
 & \hat Z_j\ar_-{\hat z_j}[l]\ar^{\nu_i^j}[d]
 & Z_i\ar_-{}[l]\ar@{=}[d] \\
Z_i\ar_{\nu_i}[r] & X
 & Z\ar^\nu[l] & Z_i\ar^{z_i}[l]
} \quad
&
\xymatrix{
\emptyset \ar[r]\ar[d]
 & X_j\ar_{x_j}[d]
 & \hat Z_j\ar_-{\hat z_j}[l]\ar^{\nu_i^j}[d]
 & \emptyset\ar[l]\ar[d] \\
Z_i\ar_{\nu_i}[r] & X
 & Z\ar^\nu[l] & Z_i\ar^{z_i}[l]
}
\end{array}
\end{split}
\end{equation}
We also consider the open and closed immersion
 $u_j:(X_j-\hat Z_j) \rightarrow (X-Z)$.

According to corollary \ref{cor:Gysin&open_closed},
 we obtain the following equalities:
$$
\nu_{ji}=z_i^* \nu^* x_{j*}, \quad 
 \partial_{i,j}=u_j^* \partial_{X,Z} z_{i*}.
$$
Then the result follows from the following computations:
\begin{align*}
z_i^* \nu^* x_{j*}
 \stackrel{(a)}=\nu_i^* x_{j*}
 \stackrel{(b)}=
&\begin{cases}
(\nu_i^j)^*  & \text{if } Z_i \subset X_j, \\
0 & \text{otherwise}.
\end{cases} \\
u_j^* \partial_{X,Z} z_{i*}
 \stackrel{(c)}=\partial_{X_j,\hat Z_j} \hat z_j^* z_{i*}
 \stackrel{(d)}=
&\begin{cases}
\partial_{X_j,\hat Z_j} (z_i^j)_*\stackrel{(e)}=\partial_{X_j-Z'_i,Z_i}
  & \text{if } Z_i \subset X_j, \\
0 & \text{otherwise}.
\end{cases}
\end{align*}
We give the following justifications for each equality:
\item[\indent (a)]: Corollary \ref{cor:assoc} ($\nu_i=\nu \circ z_i$).
\item[\indent (b)]: Proposition \ref{prop:gysin_transversal} applied
 to the first square of the respective commutative diagram
 of \eqref{eq:pf:prop:add_Gysin}
 corresponding to the each respective case.
\item[\indent (c)]: Theorem \ref{thm:assoc} applied to
 the second cartesian square of \eqref{eq:pf:prop:add_Gysin}.
\item[\indent (d)]: Proposition \ref{prop:gysin_transversal} applied 
 to the third square of the respective commutative diagram 
 of \eqref{eq:pf:prop:add_Gysin}
 corresponding to each respective case.
\item[\indent (e)]: Proposition \ref{prop:gysin_transversal}.
\end{proof}


\section{Gysin morphism}

In this section, motives are considered in the category
 $\dmgm$.

\subsection{Construction}

\subsubsection{Preliminaries}

\begin{lm}
\label{lm:double_proj_iso}
Let $X$ be a smooth scheme, $P/X$ and $Q/X$ be projective bundles of 
respective dimensions $n$ and $m$.
We consider $\lambda_P$ (resp. $\lambda_Q$) the canonical dual line
bundle on $P$ (resp. $Q$) and $\lambda'_P$
(resp. $\lambda'_Q$) its pullback on $P \times_X Q$. 
Let $p:P \times_X Q \rightarrow X$ be the canonical projection.

Then, the morphism
 $\sigma:M(P \times_X Q) \longrightarrow
\bigoplus_{i,j} M(X)(i+j)[2(i+j)]$
given by the formula
$$
\sigma=\sum_{0 \leq i \leq n, \, 0 \leq j \leq m} 
 \chern 1 {\lambda'_P}^i
 \ecupp \chern 1 {\lambda'_Q}^j
 \ecupp p_*
$$
is an isomorphism.
\end{lm}
\begin{proof}
As $\sigma$ is compatible with pullback, 
we can assume using property \textbf{(MV)} of Proposition \ref{properties_motx}
that
$P$ and $Q$ are trivializable projective bundles.
Using the invariance of $\sigma$ under 
automorphisms of $P$ or $Q$,
we can assume that $P$ and $Q$ are trivial projective bundles.
From the definition of $\sigma$, we are reduced to the
case $X=\spec k$. Then, $\sigma$ is just the tensor product 
of the two projective bundle isomorphisms 
(cf paragraph \ref{proj_bdle_iso})
for $P$ and $Q$.
\end{proof}

The following proposition is the key point in the definition of the
Gysin morphism for a projective morphism.
\begin{prop}
\label{prop:strong_normalisation_cond}
Let $X$ be a smooth scheme,
 $p:P \rightarrow X$ be a projective bundle of rank $n$
  and $s:X \rightarrow P$ a section of $p$.

Then, the composite map
$\mot X\dtwist n \xrightarrow{\lef n P} \mot P
  \xrightarrow{s^*} \mot X\dtwist{n}$
is the identity.\footnote{In fact, this result holds in the effective category
 $\dmgme$ as the proof will show.}
\end{prop}
\begin{proof}
In this proof, we work in the category $\dmgme$.

Let $\eta_P(X)$ be the motivic fundamental class associated with $s$
 (see Definition \ref{df:motivic_fdl_class}).
 According to Lemma \ref{lm:Gysin&fdl_class_retraction}, we obtain:
$s^*=\eta_P(X) \ecupp_P p_*$.

Let $E/X$ be the vector bundle on $X$ such that $P=\PP(E)$.
Let $\lambda$ be the canonical dual line bundle on $P$.
If we consider the line bundle $L=s^{-1}(\lambda^\vee)$ on $X$,
 the section $s$ corresponds uniquely to a monomorphism
 $L \rightarrow E$ of vector bundles on $P$.
 We consider the following vector bundle on $P$:
$$
F=\lambda \otimes p^{-1}(E/L).
$$
Then the canonical morphism:
$$
\lambda^\vee \rightarrow p^{-1}(E) \rightarrow p^{-1}(E/L)
$$
made by the canonical inclusion and the canonical projection
induces a section $\sigma$ of $F/P$ which is transversal 
to the zero section $s_0^F$ of $F/P$ and such that the following
square is cartesian:
$$
\xymatrix@=16pt{
X\ar^s[r]\ar[d] & P\ar^\sigma[d] \\
P\ar^{s_0^F}[r] & F.
}
$$
Thus, according to Lemma \ref{lm:Gysin&fdl_class},
 we get: $\eta_P(X)=\mathfrak c_n(F)$.

The result now follows from the computation of 
 the top Chern class $c_n(F)$ in $CH^n(P)$
 and Lemma \ref{lm:key_computation_cycle_class}.
%
\end{proof}

\begin{rem} \label{strong_normalization}
As a corollary, we obtain the following reinforcement of Proposition
\ref{prop:purity}, more precisely of the normalization condition
for the purity isomorphism~:

Let $X$ be a smooth scheme, $P/X$ be a projective bundle of rank $n$, 
and $s:X \rightarrow P$ be a section of $P/X$.
Then, the purity isomorphism $\pur{P,s(X)}$ is
 the inverse isomorphism of the composition
$$
\mot X\dtwist n \xrightarrow{\lef n P} \mot P
 \xrightarrow{(1)} \motx P {s(X)}
$$
where $(1)$ is the canonical map.
\end{rem}

\subsubsection{Gysin morphism of a projection}

The following definition will be a particular case of 
Definition \ref{gysin}.
\begin{df}
\label{Gysin_of_proj}
Let $X$ be a smooth scheme, $P$ be a projective bundle of rank $n$ 
over $X$ and $p:P \rightarrow X$ be the canonical projection.

Using the notation of \eqref{eq:lefschetz4proj}, we put:
$$p^*=\lef n P(-n)[-2n]:\mot X \rightarrow \mot P(-n)[-2n]$$
and call it the Gysin morphism of $p$.
\end{df}
\begin{lm}
\label{funct_gysin_proj}
Let $P$, $Q$ be projective bundles over a smooth scheme $X$
 of respective ranks $n$, $m$. Consider the following projections~:
$$
\xymatrix@R=-1pt@C=30pt{
& P\ar^/-1pt/p[rd] & \\
P \times_X Q\ar_/6pt/{p'}[rd]\ar^/6pt/{q'}[ru] & & X \\
& Q\ar_/-1pt/q[ru] &
}
$$

Then, the following diagram is commutative~:
$$
\xymatrix@R=-1pt@C=30pt{
& \mot P\dtwist{-m}\ar^/-9pt/{q^{\prime*}}[rd] & \\
\mot X\ar_/-2pt/{q^*}[rd]\ar^/-2pt/{p^*}[ru]
 & & \mot{P \times_X Q}\dtwist{-n-m} \\
& \mot Q\dtwist{-n}\ar_/-9pt/{p^{\prime*}}[ru] &
}
$$
\end{lm}
\begin{proof}
Indeed, using the compatibility of the motivic Chern class with pullback
 (cf \ref{motivic_chern}),
 we see that both composite morphisms $q^{\prime *}p^*$ and $p^{\prime *}q^*$ are equal 
 (up to twist and suspension) to the composite
$$
\mot X\dtwist{n+m}
 \rightarrow \bigoplus_{i \leq n, j \leq m} \mot X\dtwist{i+j}
 \rightarrow \mot{P \times_X Q},
$$
where the first arrow is the obvious split monomorphism and the second
arrow is the inverse isomorphism to the one constructed in
Lemma \ref{lm:double_proj_iso}.
\end{proof}

\subsubsection{General case}

The following lemma is all we need to finish the construction
of the Gysin morphism of a projective morphism~:
\begin{lm}
\label{lm:indep_facto}
Consider a commutative diagram
$$
\xymatrix@R=-1pt@C=14pt{
& P\ar^p[rd] & \\
Y\ar_j[rd]\ar^i[ru] & & X \\
& Q\ar_q[ru] &
}
$$
where $X$ and $Y$ are smooth schemes, 
$i$ (resp. $j$) is a closed immersion of codimension $n+d$ (resp. $m+d$), 
$P$ (resp. $Q$) is a projective bundle over $X$ of dimension $n$ (resp. $m$)
 with projection $p$ (resp. $q$).

Then, the following diagram is commutative
\begin{equation}
\label{lm:indep_facto_eq1}
\begin{split}
\xymatrix@R=1pt@C=35pt{
& \mot{P)}\dtwist{m}
   \ar^/-7pt/{i^*}[rd] & \\
\mot X\dtwist{n+m}\ar_{q^*}[rd]
                  \ar^{p^*}[ru]
 & & \mot Y\dtwist{n+m+d}. \\
& \mot{Q}\dtwist n\ar_/-7pt/{j^*}[ru] &
}
\end{split}
\end{equation}
\end{lm}
\begin{proof}
Considering the diagonal embedding 
$Y \xrightarrow{(i,j)} P \times_X Q$, we divide diagram 
\eqref{lm:indep_facto_eq1} into three parts:
$$
\xymatrix@R=16pt@C=42pt{
& \mot{P}\dtwist m\ar^{i^*}[rd]
   \ar_{p^{\prime*}}[d]
 & \\
\mot X\dtwist{n+m}\ar@/_10pt/_/-4pt/{q^*}[rd]
                  \ar@/^10pt/^/-4pt/{p^*}[ru]
  \ar@{}|/-2pt/{_{(1)}}[r]
 & \mot{P \times_X Q}\ar|-{(i,j)^*}[r]
     \ar@{}^/-30pt/{_{(2)}}[ru]\ar@{}_/-30pt/{_{(3)}}[rd]
 & \mot Y\dtwist{n+m+d}. \\
& \mot{Q}\dtwist n\ar_{j^*}[ru]
   \ar^{q^{\prime*}}[u] 
 & \\
}
$$
The commutativity of part (1) is Lemma \ref{funct_gysin_proj}.
The commutativity of part (2) and that of part (3) are equivalent
 to the case $X=Q$, $q=1_X$ -- and thus $m=0$.

Assume we are in this case.
We introduce the following morphisms where the square (*) is cartesian
 and $\gamma$ is the graph of the $X$-morphism $i$:
$$
\xymatrix@R=4pt@C=26pt{
& P_Y\ar^{p'}[r]\ar_{j'}[dd]\ar@{}|{(*)}[rdd] & Y\ar^j[dd] \\
Y\ar^\gamma[ru]\ar_i[rd] & & \\
& P\ar|p[r] & X
}
$$
Note that $\gamma$ is a section of $p'$.
Thus, Proposition  \ref{prop:strong_normalisation_cond}
 gives: $\gamma^* p^{\prime*}=1$,
and we reduce the commutativity of the diagram \eqref{lm:indep_facto_eq1}
 to that of the following one:
$$
\xymatrix@R=4pt@C=26pt{
& \mot{P_Y}\dtwist{d}\ar@{}|{(5)}[rdd]\ar_-{\gamma^*}[ld]
 & \mot Y\dtwist{n+d}\ar_-{p^{\prime*}}[l] \\
\mot Y\dtwist{n+d}\ar@{}|/18pt/{(4)}[r] & & \\
& \mot P\ar@{}|/-3pt/{j^{\prime*}}[uu]\ar|/-3pt/{\phantom{\big(}}[uu]\ar^-{i^*}[lu]
 & \mot X\dtwist n\ar|-{p^*}[l]\ar_-{j^*}[uu]
}
$$
Then commutativity of part (4) is Corollary \ref{cor:assoc}
and that of part (5) follows from Lemma 
\ref{lm:proj&Gysin_mot}.
\end{proof}

Let $f:Y \rightarrow X$ be a projective morphism between smooth
schemes. Following the terminology of Fulton (see \cite[\S 6.6]{Ful}),
 we say that $f$ has codimension $d$ if it can be factored into a
closed immersion $Y \rightarrow P$ of codimension $e$ followed by 
the projection $P \rightarrow X$ of a projective bundle of dimension $e-d$.
In fact, the integer $d$ is uniquely determined
(cf \emph{loc.cit.} appendix B.7.6).
Using the preceding lemma, we can finally introduce 
the general definition~:
\begin{df}
\label{gysin}
Let $X$, $Y$ be smooth schemes
 and $f:Y \rightarrow X$ be a projective morphism of codimension $d$.

We define the Gysin morphism associated with $f$ in $\dmgm$
$$
f^*:\mot X \rightarrow \mot Y\dtwist d
$$
by choosing a factorisation of $f$ into $Y \xrightarrow i P
\xrightarrow p X$ where $i$ is a closed immersion of pure codimension
$n+d$ and $p$ is the projection of a projective bundle of rank $n$,
 and putting~:
$$
f^*=\left\lbrack \mot X\dtwist{n} \xrightarrow{\lef n P} \mot P
 \xrightarrow{i^*} \mot Y\dtwist{n+d} \right\rbrack \dtwist{-n},
$$
definition which does not depend upon the choices made according
to the previous lemma.
\end{df}

\begin{rem} \label{Gysin&pushout}
In \cite[3.11]{Deg9}, we prove that
the Gysin morphism of a projective morphism $f$ induces the
usual pushout on the part of motivic cohomology corresponding to
Chow groups.
\end{rem}


\subsection{Properties}

\subsubsection{Functoriality}

\begin{prop}
\label{Gysin_funct}
Let $X$, $Y$, $Z$ be smooth schemes and 
$Z \xrightarrow g Y \xrightarrow f X$ be projective morphisms of
respective codimensions $m$ and $n$.

Then, in $\dmgm$, we get the equality~:
 $g^* \circ f^*=(fg)^*$. 
\end{prop}
\begin{proof}
We first choose projective bundles $P$, $Q$
over $X$, of respective dimensions $s$ and $t$, fitting into the following
diagram with $R=P \times_X Q$ and $Q_Y=Q \times_X Y$:
$$
\xymatrix@C=16pt@R=6pt{
&& Q\ar^q@/^12pt/[rrddd] && \\
& & R\ar^{q'}[rd]\ar_{p'}[u] & & \\
& Q_Y\ar^{q''}[rd]\ar^{i'}[ru] & & P\ar|p[rd] & \\
Z\ar|g[rr]\ar|k[ru]\ar^j@/^12pt/[rruuu] && Y\ar|f[rr]\ar^i[ru] && X.
}
$$
The prime exponent of a symbol indicates that the morphism is deduced
by base change from the morphism with the same symbol. We then have to
prove that the following diagram of $\dmgm$ commutes~:
$$
\xymatrix@C=4pt@R=12pt{
&\ar|/3pt/{(2)}@{}[dd] 
 & \mot Q\dtwist t\ar^{{p'}^*}[d]\ar^{j^*}@/^34pt/[rrddd]
 &\ar|/3pt/{(3)}@{}[dd] & \\
& & \mot{R}\dtwist{s+t}\ar_{{i'}^*}[rd]\ar|{(1)}@{}[dd]
  & & \\
& \mot P\dtwist s\ar_{{q'}^*}[ru]\ar_{i^*}[rd]
 & & \mot{Q_Y}\dtwist{n+t}\ar_{k^*}[rd] & \\
\mot X\ar_{p^*}[ru]\ar^{q^*}@/^38pt/[rruuu]
 && \mot Y\dtwist n\ar_{{q''}^*}[ru]
 && \mot Z\dtwist{n+m}.
}
$$
The commutativity of part (1) is a corollary of Lemma
\ref{lm:proj&Gysin_mot}, that of part (2) is 
Lemma \ref{funct_gysin_proj} and that of part (3) follows
from Lemma \ref{lm:indep_facto} and Corollary \ref{cor:assoc}.
\end{proof}

\subsubsection{Projection formula and excess of intersection}

From Definition \ref{gysin}
 and Proposition \ref{prop:gysin_transversal}
  we directly obtain the following proposition~:
\begin{prop}
\label{trivial_proj_form_Gysin}
Consider a cartesian square of smooth schemes
\begin{equation} \label{eq:trivial_proj_form_Gysin}
\xymatrix@=10pt{
T\ar_q[d]\ar^g[r] & Z\ar^p[d] \\
Y\ar^f[r] & X
}
\end{equation}
such that $f$ and $g$ are projective morphisms of the same codimensions.

Then, the relation $f^*p_*=q_* g^*$ holds in $\dmgm$.
\end{prop}
%

\num Consider now a cartesian square of shape \eqref{eq:trivial_proj_form_Gysin}
such that $f$ (resp. $g$)
is a projective morphism of codimension $m$ (resp. $m$).
Then $m \leq n$ and we call $e=n-m$ the \emph{excess of dimension}
 attached with \eqref{eq:trivial_proj_form_Gysin}.

We can also associate with the above square a vector bundle $\xi$ of rank $e$,
called the \emph{excess bundle}.
Choose $Y \xrightarrow i P \xrightarrow \pi X$ a factorisation of $f$ 
such that $i$ is a closed immersion of codimension $r$ 
and $\pi$ is the projection of a projective bundle of dimension $s$.
We consider the following cartesian squares:
\begin{equation*}
\xymatrix@=10pt{
T\ar_q[d]\ar^{i'}[r] & Q\ar^{\pi'}[r]\ar[d] & Z\ar^p[d] \\
Y\ar^i[r] & P\ar^\pi[r] & X
}
\end{equation*}
Then $N_TQ$ is a sub-vector bundle of $q^{-1}N_YP$
 and we put $\xi=q^{-1}N_YP/N_TQ$. This
definition is independent of the choice of $P$ (see \cite{Ful}, 
 proof of prop. 6.6). 

The following proposition is now
a straightforward consequence of Definition \ref{gysin}
 and the second case of Proposition \ref{prop:gysin_transversal}~: 
\begin{prop}
\label{excess}
Consider the above notations.

Then, the relation
 $f^*p_*=\big(\chern e \xi \ecupp q_*\dtwist m\big) \circ g^*$
holds in $\dmgm$.
\end{prop}

\subsubsection{Compatibility with the Gysin triangle}

\begin{prop}
\label{Gysin_morph&Gysin_tri}
Consider a topologically cartesian square of smooth schemes
$$
\xymatrix@=8pt{
T\ar^j[r]\ar_g[d] & Y\ar^f[d] \\
Z\ar^i[r] & X
}
$$
such that $f$ and $g$ are projective morphisms, $i$ and $j$ are
closed immersions. Put $U=X-Z$, $V=Y-T$
 and let $h:V \rightarrow U$ be the projective morphism
  induced by $f$. Let $n$, $m$, $p$, $q$ be respectively 
the relative codimensions of $i$, $j$, $f$, $g$.

Then the following diagram is commutative
$$
\xymatrix@C=9pt@R=15pt{
\mot{V}\dtwist p\ar[r] 
 & \mot{Y}\dtwist p\ar^-{j^*}[rr]
 && \mot{T}\dtwist{m+p}\ar^-{\partial_{Y,T}}[rrr]
 &&& \mot{V}\dtwist p[1] \\
 \mot{U}\ar[r]\ar^{h^*}[u]
 & \mot{X}\ar^-{i^*}[rr]\ar^{f^*}[u]
 && \mot{Z}\dtwist{n}\ar^-{\partial_{X,Z}}[rrr]\ar_{g^*\dtwist n}[u]
 &&& \mot{U}[1]\ar_{h^*}[u], \\
}
$$
where the two lines are the obvious Gysin triangles.
\end{prop}
\begin{proof}
Use the definition of the Gysin morphism
 and apply Lemma \ref{lm:proj&Gysin_mot}, Theorem \ref{thm:assoc}.
%
\end{proof}

\subsubsection{Gysin morphisms and transfers in the \'etale case}

\num \label{transpose}
In \cite{Deg5bis}, paragraphs 1.1 and 1.2 we have introduced another Gysin morphism for 
a finite equidimensional morphism $f:Y \rightarrow X$. 
Indeed, the transpose of the graph of $f$ gives
 a finite correspondence $\tra f$ from $X$ to $Y$
 which induces a morphism $\tra f_*:\mot X \rightarrow \mot Y$ in $\dmgm$.
\begin{prop}
\label{prop:Gysin=transpose}
Let $X$ and $Y$ be smooth schemes, and
$f:Y \rightarrow X$ be an \'etale cover.

Then, $f^*=\tra f_*$.
\end{prop}
\begin{proof}
Consider the cartesian square of smooth schemes
$$
\xymatrix@=10pt{
Y \times_X Y\ar^-{g}[r]\ar_{f'}[d] & Y\ar^f[d] \\
Y\ar^f[r] & X.
}
$$
We first prove that $\tra f'_* f^*=g^* \, \tra f_*$.
Choose a factorisation $Y \xrightarrow i P \xrightarrow \pi X$ of 
$f$ into a closed immersion and the projection of a projective bundle. 
The preceding square can be divided into two squares
$$
\xymatrix@=10pt{
{Y \times_X Y}\ar^-j[r]\ar_{f'}[d]
 & {P \times_X Y}\ar^-q[r]\ar_{f''}[d]
 &  Y\ar^f[d] \\
 Y\ar^i[r] &  P\ar^\pi[r] &  X.
}
$$
The assertion then follows from the commutativity of the following 
diagram.
$$
\xymatrix@C=12pt@R=14pt{
\mot{Y \times_X Y}\ar@{}|{_{(1)}}[rd]
 & \mot{P \times_X Y}\ar_/-2pt/{j^*}[l]\ar@{}|{_{(2)}}[rd]
 & \mot Y\ar_/-10pt/{q^*}[l] \\
\mot Y\ar^{\tra f'_*}[u]
 & \mot P\ar^{\tra f''_*}[u]\ar^/-2pt/{i^*}[l]
 & \mot X\ar_{\tra f_*}[u]\ar^/-4pt/{p^*}[l].
}
$$
The commutativity of part (1) follows from \cite{Deg5bis}, prop. 2.5.2 
(case 1) and that of part (2) from \cite{Deg5bis}, prop. 2.2.15 (case 3).

Then, considering the diagonal immersion 
$Y \xrightarrow \delta Y \times_X Y$, it suffices to prove in view of
Proposition \ref{Gysin_funct} that $\delta^* \circ \tra f'_*=1$. 
As $Y/X$ is {\'e}tale, $Y$ is a connected component of $Y \times_X Y$. 
Thus, $\mot Y$ is a direct factor of $\mot{Y \times_X Y}$. 
Then, according to corollary \ref{cor:Gysin&open_closed},
 $\delta^*$ is the canonical projection on this direct factor.
One can easily see that $\tra f'_*$ is the canonical inclusion
 and this concludes.
\end{proof}


\subsection{Duality pairings, motive with compact support}

\begin{paragr} \label{num:df_duality}
We first recall the abstract definition of duality in monoidal
categories.
Let $\C$ be a symmetric monoidal category with product $\otimes$ and unit $\unit$.
An object $X$ of $\C$ is said to be \emph{strongly dualizable}
if there exists an object $\dual{X}$ of $\C$ and two maps
$$
\eta:\unit \rightarrow \dual{X} \otimes {X}, 
\quad \epsilon:{X} \otimes \dual{X} \rightarrow \unit
$$
such that the following diagrams commute:
$$
\xymatrix{
X\ar[r]^(.3){X\otimes\eta}\ar[dr]_{1_X}
 & X\otimes\dual{X}\otimes X\ar[d]^{\epsilon\otimes X}
 &
&\dual{X}\ar[r]^(.35){\eta\otimes\dual{X}}\ar[dr]_{1_{\dual{X}}}
&\dual{X}\otimes X\otimes\dual{X}\ar[d]^{\dual{X}\otimes\epsilon}\\
&X&&&\dual{X}
}
$$
The object $X^*$ is called a \emph{strong dual} of $X$.
For any objects $Y$ and $Z$ of $\C$, we then have a canonical bijection
$$
\Hom_\C(Z\otimes X,Y)\simeq\Hom_\C(Z,\dual{X}\otimes Y).
$$
In other words, $\dual{X}\otimes Y$ is the internal $\Hom$
of the pair $(X,Y)$ for any $Y$. In particular, such a dual is unique
up to a canonical isomorphism. If $\dual{X}$ is a strong dual of $X$, then
$X$ is a strong dual of $\dual{X}$. \\
\indent Suppose $\C$ is a closed symmetric monoidal
triangulated category. Denote by $\sHom$ its internal Hom.
For any objects $X$ and $Y$ of $\C$
the evaluation map
$$X\otimes\sHom(X,\unit)\rightarrow\unit$$
tensored with the identity of $Y$ defines by adjunction a map
$$
\sHom(X,\unit)\otimes Y\rightarrow \sHom(X,Y).
$$
The object $X$ is strongly dualizable if and only if this map is an isomorphism
for all objects $Y$ in $\C$. In this case indeed,
$\dual{X}=\sHom(X,\unit)$.
\end{paragr}

\begin{paragr} \label{par:duality}
Let $X$ be a smooth projective $k$-scheme of pure dimension $n$
and denote by
$p:X \rightarrow \spec k$ the canonical projection, 
$\delta:X \rightarrow X \times_k X$ the diagonal embedding.

Then we can define morphisms
\begin{align*}
\eta: \ & \ZZ \xrightarrow{p^*} M(X)(-n)[-2n]
 \xrightarrow{\delta_*} M(X)(-n)[-2n] \otimes M(X) \\
\epsilon: \ & M(X) \otimes M(X)(-n)[-2n]
 \xrightarrow{\delta^*} M(X) \xrightarrow{p_*} \ZZ.
\end{align*}
One checks easily using the properties of the Gysin morphism these maps
turn $M(X)(-n)[-2n]$ into the dual of $M(X)$. We thus have obtained~:
\begin{prop} \label{duality}
Let $X/k$ be a smooth projective scheme.

Then the couple of morphisms $(\eta,\epsilon)$ defined above 
is a duality pairing. Thus $M(X)$ is strongly dualizable with dual
$M(X)(-n)[-2n]$.
\end{prop}

\begin{rem} \label{duality&Chow}
Using this duality in conjunction with the isomorphism \eqref{motivic_coh&Chow}, 
 we obtain for smooth projective schemes $X$ and $Y$, 
 $d$ being the dimension of $Y$, a canonical map:
\begin{align*}
CH^d(X \times Y) 
 &\simeq \Hom_{\dmgme}(\mot X \otimes \mot Y,\ZZ(d)[2d]) \\
 & \rightarrow \Hom_{\dmgm}(\mot X \otimes \mot Y,\ZZ(d)[2d]) \\
  & \qquad =\Hom_{\dmgm}(\mot X,\mot Y).
\end{align*}
As the isomorphism \eqref{motivic_coh&Chow} is compatible with products and pullbacks,
 we check easily this defines a monoidal functor from Chow motives
 to mixed motives obtaining a new construction of the stable version of the functor
 which appears in \cite[chap. 5, 2.1.4]{FSV}.
Recall that the cancellation theorem of Voevodsky \cite{V2} implies this is
 a full embedding.
\end{rem}

Note the Gysin morphism $p^*:\ZZ(n)[2n] \rightarrow M(X)$
defines indeed a homological class $\eta_X$ 
in $H^\mathcal M_{2n,n}(X)=\Hom_{\dmgm}(\ZZ(n)[2n],M(X))$.

The duality above induces an isomorphism
$$
H^{p,q}_\mathcal M(X) \rightarrow H_{p-2n,q-n}^\mathcal M(X)
$$
which is by definition the cap-product by $\eta_X$. 
Thus our duality pairing implies the classical form of Poincar{\'e} 
duality and the class $\eta_X$ is the fundamental class of $X$.
\end{paragr}

\begin{paragr} \label{par:compact}
The last application of this section 
 uses the stable version of the category of motivic
 complexes as defined in \cite[7.15]{CD1} and denoted by $\DM$.
 Remember it is a triangulated symmetric monoidal category.
 Moreover, there is a canonical monoidal fully faithful functor 
 $\dmgm \rightarrow \DM$ (see \cite[10.1.4]{CD3}).
The idea of the following definition comes from \cite[2.6.3]{CD2}:
\end{paragr} 
\begin{df} \label{df:motif_c}
Let $X$ be a smooth scheme of dimension $d$.

We define the motive with compact support of $X$ as the object
 of $\DM$
$$
M^c(X)=\mathrm R \sHom_{\DM}(M(X),\ZZ(d)[2d]).
$$
\end{df}
This motive with compact support satisfies the following properties:
\begin{enumerate}
\item[(i)] For any morphism $f:Y \rightarrow X$ of relative dimension $n$
 between smooth schemes, the usual functoriality of motives induces:
$$
f^*:M^c(X)(n)[2n] \rightarrow M^c(Y).
$$
\item[(ii)] For any projective morphism $f:Y \rightarrow X$ between
 smooth schemes, the Gysin morphism of $f$ induces:
$$
f_*:M^c(Y) \rightarrow M^c(X).
$$
\item[(iii)] Let $i:Z \rightarrow X$ be a closed immersion
 between smooth schemes, and $j$ the complementary open immersion.
 Then the Gysin triangle associated with $(X,Z)$ induces
 a distinguished triangle:
$$
M^c(Z) \xrightarrow{i_*} M^c(X) \xrightarrow{j^*} M^c(U)
 \xrightarrow{\partial_{X,Z}'} M^c(Z)[1].
$$
\item[(iv)] If $X$ is a smooth $k$-scheme of relative dimension $d$,
 $p$ its structural morphism and $\delta$ its diagonal embedding, 
 the composite morphism
$$
M(X) \otimes M(X) \xrightarrow{\delta^*} M(X)(d)[2d]
 \xrightarrow{p_*} \ZZ(d)[2d]
$$
induces a map
$$
\phi_X:M(X) \rightarrow M^c(X)
$$
which is an isomorphism when $X$ is projective (cf \ref{duality}).
Moreover, for any open immersion $j:U \rightarrow X$,
 $j^* \circ \phi_X \circ j_*=\phi_U$
 (this follows easily from \ref{trivial_proj_form_Gysin}).
\end{enumerate}

\begin{rem}
Note also that the formulas we have proved for
 the Gysin morphism or the Gysin triangle correspond
 to formulas involving the data (i), (ii) or (iii)
 of motives with compact support.
\end{rem}

\begin{num}
Consider a smooth scheme $X$ of pure dimension $d$.
According to Definition \ref{df:motif_c},
 as soon as $M(X)$ admits a strong dual $M(X)^\vee$
 in $\DM$, we get a canonical isomorphism:
\begin{equation} \label{eq:compact&dual}
M^c(X)=M(X)^\vee(d)[2d].
\end{equation}
The same remark can be applied if we work in $\DM \otimes \QQ$.
Recall that duality is known in the following cases
 (it follows for example from the main theorem of \cite{Riou}):
\end{num}
\begin{prop} Let $X$ be a smooth scheme of dimension $d$.
\begin{enumerate}
\item Assume $k$ admits resolution of singularities. \\
Then $M(X)$ is strongly dualizable in $\dmgm$.
\item In any case, 
 $M(X) \otimes \QQ$ is strongly dualizable in  $\dmgm \otimes \QQ$.
\end{enumerate}
\end{prop}
Recall that Voevodsky has defined a motive with compact
 support (even without the smoothness assumption). It
 satisfies all the properties listed above except that
 (i) and (iii) requires resolution of singularities.
Then according to the preceding proposition and formula
 \eqref{eq:compact&dual}, our definition
 agrees with that of Voevodsky if resolution of singularities
 holds over $k$ (apply \cite[chap. 5, th. 4.3.7]{FSV}).
This implies in particular that $M^c(X)$ is in $\dmgm$ or,
 in the words of Voevodsky,
 it is \emph{geometric}.
Moreover, we know from the second case of the preceding
proposition that
 $M^c(X) \otimes \QQ$ is always geometric.



\bibliographystyle{alpha}
\bibliography{gysin}

\end{document}